\documentclass[review,3p,square,11pt]{elsarticle}

\usepackage{amssymb}
\usepackage{amsmath}
\usepackage{amsthm}
\usepackage{bbm}
\usepackage{enumitem}

\usepackage{fancyhdr}
\usepackage{hyperref}
\hypersetup{%
  pdftitle={BSDEs},%
  pdfsubject={BSDEs},%
  pdfauthor={ShengJun FAN, Ying Hu},%
  pdfkeywords={BSDEs},%
  pdfstartview=FitH,%
  CJKbookmarks=true,%
  bookmarksnumbered=true,%
  bookmarksopen=true,%
  colorlinks=true, linkcolor=blue, urlcolor=blue, citecolor=blue, %
}

\usepackage{cleveref}
\crefname{thm}{Theorem}{Theorems}
\crefname{pro}{Proposition}{Propositions}
\crefname{lem}{Lemma}{Lemmas}
\crefname{rmk}{Remark}{Remarks}
\crefname{cor}{Corollary}{Corollaries}
\crefname{dfn}{Definition}{Definitions}
\crefname{ex}{Example}{Examples}
\crefname{section}{Section}{Sections}
\crefname{subsection}{Subsection}{Subsections}

\newcommand{\eps}{\varepsilon}

\newcommand{\To}{\rightarrow}
\newcommand{\as}{{\rm d}\mathbb{P}\times{\rm d} t-a.e.}
\newcommand{\ass}{{\rm d}\mathbb{P}\times{\rm d} s-a.e.}
\newcommand{\ps}{\mathbb{P}-a.s.}

\newcommand{\F}{\mathcal{F}}
\newcommand{\E}{\mathbb{E}}

\newcommand{\s}{\mathcal{S}}

\newcommand{\mcal}{\mathcal{M}}

\newcommand{\hcal}{\mathcal{H}}

\newcommand{\T}{[0,T]}

\newcommand{\R}{{\mathbb R}}

\newcommand{\RE}{\forall}

\newcommand {\Dis}{\displaystyle}

\newtheorem{thm}{Theorem}[section]
\newtheorem{lem}[thm]{Lemma}

\newtheorem{rmk}[thm]{Remark}

\journal{}
\begin{document}
\begin{frontmatter}

\title{{Multi-dimensional backward stochastic differential equations of diagonally quadratic generators: the general result}\tnoteref{found}}
\tnotetext[found]{Shengjun Fan is supported by the State Scholarship Fund from the China Scholarship Council (No. 201806425013). Ying Hu is partially supported by Lebesgue center of mathematics ``Investissements d'avenir" program-ANR-11-LABX-0020-01, by CAESARS-ANR-15-CE05-0024 and by MFG-ANR-16-CE40-0015-01. Shanjian Tang is supported by National Science Foundation of China (No. 11631004).
\vspace{0.2cm}}


\author{Shengjun Fan\vspace{-0.5cm}\corref{cor1}}
\author{\ \ \ \ Ying Hu\corref{cor2}}
\author{\ \ \ \ Shanjian Tang$^\dag$\corref{cor3}}

\cortext[cor1]{\ School of Mathematics, China University of Mining and Technology, Xuzhou 221116, China. E-mail: f\_s\_j@126.com \vspace{0.2cm}}

\cortext[cor2]{\ Univ. Rennes, CNRS, IRMAR-UMR6625, F-35000, Rennes, France. E-mail: ying.hu@univ-rennes1.fr\vspace{0.2cm}}

\cortext[cor3]{${}^\dag$Department of Finance and Control Sciences, School of Mathematical Sciences, Fudan University, Shanghai 200433, China. E-mail: sjtang@fudan.edu.cn}

\begin{abstract}
This paper is devoted to a general solvability of a
multi-dimensional backward stochastic differential equation (BSDE) of a diagonally quadratic generator $g(t,y,z)$, by relaxing  the assumptions of \citet{HuTang2016SPA} on the generator and terminal value. More precisely, the generator $g(t,y,z)$ can have more general growth and continuity in $y$ in the local solution; while in the global solution,  the generator $g(t,y,z)$ can have  a  skew sub-quadratic but  in addition  ``strictly and diagonally" quadratic growth in   the second unknown variable $z$, or the terminal value can be unbounded but the generator $g(t,y,z)$  is ``diagonally dependent" on the second unknown variable $z$ (i.e., the $i$-th component $g^i$ of the generator $g$ only depends on the $i$-th row $z^i$ of the  variable $z$ for each  $i=1,\cdots,n$ ). Three new results are established on the local and global solutions  when the terminal value is bounded and  the generator $g$ is subject to some general assumptions. When  the terminal value is unbounded but  is of  exponential moments of arbitrary order, an existence and uniqueness result  is given under the assumptions that  the generator $g(t,y,z)$ is Lipschitz continuous in the first unknown variable $y$, and varies with the second unknown variable $z$  in a ``diagonal'' , ``component-wisely convex or concave",  and ``quadratically growing" way,  which seems to be the first  general solvability of
system of quadratic BSDEs with unbounded terminal values. This generalizes and strengthens some existing results via some new ideas.
\vspace{0.2cm}
\end{abstract}

\begin{keyword}
Multi-dimensional BSDE \sep diagonally quadratic generator \sep convex generator \sep \\ \hspace*{1.95cm}BMO martingale \sep unbounded terminal value. \vspace{0.2cm}

\MSC[2010] 60H10\vspace{0.2cm}
\end{keyword}

\end{frontmatter}
\vspace{-0.4cm}

\section{Introduction}
\label{sec:1-Introduction}
\setcounter{equation}{0}

Fix a terminal time $T\in (0,+\infty)$ and two positive integers $n$ and $d$.
Let $(B_t)_{t\in\T}$ be a $d$-dimensional standard Brownian motion defined on some complete probability space $(\Omega, \F, \mathbb{P})$, and $(\F_t)_{t\in\T}$ be the augmented natural filtration generated by the standard Brownian motion $B$.
Consider the following backward stochastic differential equation (BSDE in short):
\begin{equation}\label{eq:1.1}
  Y_t=\xi+\int_t^T g(s,Y_s,Z_s){\rm d}s-\int_t^T Z_s {\rm d}B_s, \ \ t\in\T,
\end{equation}
where the terminal value $\xi$ is an $\F_T$-measurable $n$-dimensional random vector, the generator function $g(\omega, t, y, z):\Omega\times\T\times\R^n\times\R^{n\times d}\To \R^n$
is $(\F_t)$-progressively measurable for each pair $(y,z)$, and the solution $(Y_t,Z_t)_{t\in\T}$ is a pair of $(\F_t)$-progressively measurable processes with values in $\R^n\times\R^{n\times d}$ which almost surely verifies BSDE \eqref{eq:1.1}.
The history of BSDEs (1.1) can be dated back to  \citet{Bismut1973JMAA} for the linear case, and to \citet{Bismut1976SICON} for a specifically structured matrix-valued nonlinear case where the matrix-valued generator contains a quadratic form of the second unknown.  In 1990, \citet{PardouxPeng1990SCL} established the existence and uniqueness result for a multidimensional ($n\ge 1$) nonlinear BSDE with a uniformly Lipschitz continuous generator. Subsequently, there has been an increasing interest in BSDEs with applications in various fields such as stochastic control, mathematical finance, partial differential equations (PDEs).

The class of BSDEs, with generators having a quadratic growth in the state variable $z$, has attracted a lot of attention in recent years. On the one hand, the existence and uniqueness theory is well developed in the scalar ($n=1$) case. \citet{Kobylanski2000AP} established the first existence and uniqueness result for scalar-valued quadratic BSDEs with bounded terminal values, and some subsequent intensive efforts can be founded in \citet{Tevzadze2008SPA}, \citet{BriandElie2013SPA}, \citet{Fan2016SPA} and \citet{LuoFan2018SD} for the bounded terminal value case, and in \citet{BriandHu2006PTRF,BriandHu2008PTRF},
\citet{DelbaenHuRichou2011AIHPPS,DelbaenHuRichou2015DCD}, \citet{BarrieuElKaroui2013AoP} and \citet{FanHuTang2019ArXiv} for the unbounded terminal value case. On the other hand, \citet{FreiDosReis2011MFE} constructed an example of multidimensional quadratic BSDE with a simple generator and a bounded terminal value to show that the equation might fail to have a global bounded solution on $\T$, which illustrates the difficulty of the quadratic part contributing to the
underlying scalar generator as an unbounded process. Moreover, it is well known that some tools used when $n=1$, like Girsanov's transform and monotone convergence, can no longer be applied when $n>1$ in most cases. Consequently, multidimensional quadratic BSDEs, the focus of the present paper, pose a great challenge. Solutions of multidimensional quadratic BSDEs with unbounded terminal values have been listed as an open problem in Peng~\cite[Section 5, page 270]{Peng1999}.

Nevertheless, motivated by their intrinsic mathematical interest and especially by diverse applications in various fields, such as nonzero-sum risk-sensitive stochastic differential games, financial price-impact models, financial market equilibrium problems for several interacting agents, and stochastic equilibria problems in incomplete financial markets, many scholars have studied systems of quadratic BSDEs in recent years. First of all, by the theory of BMO (bounded in mean oscillation) martingales and using the contract mapping argument, \citet{Tevzadze2008SPA} proved a general existence and uniqueness result for multi-dimensional quadratic BSDEs when the terminal value is small enough in the supremum norm, which has inspired subsequent works under some different types of ``smallness" assumptions on the terminal value and the generator, see for example \citet{Frei2014SPA}, \citet{Kardaras2015ArXiv}, \citet{JamneshanKupperLuo2017ECP}, \citet{KramkovPulido2016AAP, KramkovPulido2016SIAM} and \citet{HarterRichou2019EJP}. We also note that some different ideas and methods have been applied in these works mentioned above. Secondly, in the Markovian setting, \citet{cheridito2015Stochastics} proved the solvability for a special system of quadratic BSDEs, and \citet{XingZitkovic2018AoP} obtained, by virtue of analytic PDE methods, the global solvability for a large class of multidimensional quadratic BSDEs with weak regularity assumptions on the terminal value and the generator. Finally, by utilizing the Girsanov transform and adopting a distinct idea from the above works, which is to search for some sufficient conditions on the generator such that the corresponding system of quadratic BSDEs admits a (unique) local or global solution for any bounded terminal values rather than some certain terminal values, \citet{cheridito2015Stochastics}, \citet{HuTang2016SPA} and \citet{Luo2019ArXiv} respectively established several existence and uniqueness results of local and global solutions for systems of BSDEs with some structured quadratic generators. More specifically, \citet{cheridito2015Stochastics} investigated system of BSDEs with projectable quadratic generators and subquadratic generators, \citet{HuTang2016SPA} addressed a kind of multi-dimensional BSDEs with diagonally quadratic generators, in which the $i$th $(i=1,\cdots,n)$ component $g^i$ of the generator $g$ has a quadratic growth only on the $i$th row of the matrix $z$, and \citet{Luo2019ArXiv} considered a class of multi-dimensional BSDEs with triangularly quadratic generators. We would like to mention that all of these results mentioned in this paragraph are obtained under the bounded terminal value condition, and up to our best knowledge, in existing literature there seems to be no positive general solvability result on the system of quadratic BSDEs with unbounded terminal values.

The present paper is the continuation and extension of \citet{HuTang2016SPA}.
Under more general assumptions on the generator and the terminal condition than those used in \cite{HuTang2016SPA}, we are devoted to the general solvability of multidimensional diagonally quadratic BSDEs. A local solution is first constructed by virtue of uniform a priori estimates on the solution of scalar-valued BSDEs and the fixed-point argument, where the terminal value is bounded and the generator $g$ can have a general growth in the variable $y$. We note that a simpler and more direct idea than that used in the proof of \cite[ Theorem 2.2 ]{HuTang2016SPA} is used to obtain the radius of the centered ball within which the constructed mapping is stable. Then, by stitching local solutions we proved two existence and uniqueness results on global solution of system of diagonally quadratic BSDEs with bounded terminal values, where the generator $g$ needs to satisfy an additional one-sided growth condition with respect to the variable $y$. In particular, we eliminate the restriction condition used in \cite{HuTang2016SPA} that the $i$th component of the generator $g$ is bounded with respect to the $j$th ($j\neq i$) row of the matrix $z$ by imposing a strictly quadratic condition on the generator $g$ (see assumption \ref{A:H5} in Section \ref{sec:2-statement of main result}). Finally,
assuming that for each $i=1,\cdots,n$, the $i$th component $g^i$ of the generator $g$ is Lipschitz continuous in the state variable $y$, depends only on the $i$th row $z^i$ of the state variable $z$, and is either convex or concave with quadratic growth in $z^i$,  utilizing the iterative algorithm together with uniform a priori estimates and the $\theta$-method, we prove existence and uniqueness of the global solution to the multidimensional diagonally quadratic BSDE with the terminal value of exponential moments of arbitrary order, which seems to be the first result on the general solvability of system of quadratic BSDEs with unbounded terminal values.

The rest of the paper is organized as follows. In Section 2, we introduce some notations used later, and state the main results of this paper. In Sections 3-5,
we respectively prove our existence and uniqueness results on the local and global solution for our multi-dimensional diagonally quadratic BSDEs with bounded and unbounded terminal values. Finally, in the Appendix we present some auxiliary results for scalar-valued quadratic BSDEs with bounded and unbounded terminal values, including existence, uniqueness and several a priori estimates.

\section{Notations and statement of main results}
\label{sec:2-statement of main result}
\setcounter{equation}{0}

\subsection{Notations\vspace{0.1cm}}

Let $a\wedge b$ and $a\vee b$ be the minimum and maximum of two real numbers $a$ and $b$, respectively. Set $a^+:=a\vee 0$   and $a^-:=-(a\wedge 0)$. Denote by ${\bf 1}_A(\cdot)$ the indicator of set $A$, and ${\rm sgn}(x):={\bf 1}_{x>0}-{\bf 1}_{x\leq 0}$.

Throughout this paper, all the processes are assumed to be $(\F_t)_{t\in\T}$-progressively measurable, and
all equalities and inequalities between
random variables and processes are understood in the sense of $\ps$ and $\as$,  respectively. The Euclidean norm is always denoted by $|\cdot|$,  and $\|\cdot\|_{\infty}$ denotes the $L^{\infty}$-norm for one-dimensional or multidimensional random variable defined on the probability space $(\Omega, \F, \mathbb{P})$.

We define  the following four Banach spaces of stochastic processes.
By $\s^p(\R^n)$ for $p\geq 1$ , we denote the totality of  all $\R^n$-valued continuous adapted processes $(Y_t)_{t\in\T}$ such that
$$\|Y\|_{{\s}^p}:=\left(\E[\sup_{t\in\T} |Y_t|^p]\right)^{1/p}<+\infty.$$
By $\s^{\infty}(\R^n)$, we denote the totality of all $Y\in \bigcap_{p\geq 1}\s^p(\R^n)$ such that
$$\|Y\|_{{\s}^{\infty}}:=\mathop{\rm ess\ sup}_{(\omega,t)}|Y_t(\omega)| =\left\|\sup_{t\in\T} |Y_t| \right\|_{\infty}<+\infty.$$
By $\hcal^p(\R^{n\times d})$ for $p\geq 1$, we denote the totality of  all $\R^{n\times d}$-valued $(\F_t)_{t\in\T}$-progressively measurable processes $(Z_t)_{t\in\T}$ such that
$$
\|Z\|_{\hcal^p}:=\left\{\E\left[\left(\int_0^T |Z_s|^2{\rm d}s\right)^{p/2}\right] \right\}^{1/p}<+\infty.
$$
By ${\rm BMO}(\R^{n\times d})$, we denote the totality of all $Z\in \hcal^2(\R^{n\times d})$ such that
$$
\|Z\|_{\rm BMO}:=\sup_{\tau}\left\|\E_{\tau}\left[\int_{\tau}^T |Z_s|^2 {\rm d}s\right]\right\|_{\infty}^{1/2}<+\infty.
$$
Here and hereafter the supremum is taken over all $(\F_t)$-stopping times $\tau$ with values in $\T$, and $\E_{\tau}$ denotes the conditional expectation with respect to $\F_\tau$.

The spaces  $\s^p_{[a,b]}(\R^n)$, $\s^{\infty}_{[a,b]}(\R^n)$, $\hcal^p_{[a,b]}(\R^{n\times d})$,  and ${\rm BMO}_{[a,b]}(\R^{n\times d})$ are identically defined for stochastic processes over the time interval $[a,b]$. We note that for $Z\in {\rm BMO}(\R^{n\times d})$, the process $\int_0^t Z_s{\rm d}B_s, t\in\T$,  is an $n$-dimensional BMO martingale.  For the theory of BMO martingales, we refer the reader  to the monograph  \citet{Kazamaki1994book}.

For $i=1,\cdots, n$, denote by $z^i$, $y^i$  and $g^i$ respectively the $i$th row of matrix $z\in\R^{n\times d}$,  the $i$th component of the vector $y\in \R^n$ and the generator $g$.

Finally, we write $Y\in \mathcal{E}(\R^n)$ if
$$\exp{(|Y|)}\in \bigcap_{p\geq 1}\s^p(\R^n),$$
and $Z\in \mcal(\R^{n\times d})$ if
$$Z\in \bigcap_{p\geq 1}\hcal^p(\R^{n\times d}).\vspace{0.2cm}$$

\subsection{Statement of the main results}

Throughout the paper, we always fix an $(\F_t)$-progressively measurable scalar-valued non-negative process $(\alpha_t)_{t\in\T}$, a deterministic nondecreasing continuous function $\phi(\cdot):[0,+\infty)\To [0,+\infty)$ with $\phi(0)=0$ and several real constants $\beta\geq 0$, $0<\bar\gamma\leq \gamma$, $\lambda\geq 0$ and $\delta\in [0,1)$.

The first main result of this paper concerns local solutions for the bounded terminal value case. We need the following three assumptions.

\begin{enumerate}
\renewcommand{\theenumi}{(H\arabic{enumi})}
\renewcommand{\labelenumi}{\theenumi}
\item\label{A:H1} For $i=1,\cdots,n$, $g^i$ satisfies that $\as$, for each $(y,z)\in \R^n\times\R^{n\times d}$,
    $$
    |g^i(\omega,t,y,z)|\leq \alpha_t(\omega)+\phi(|y|)+\frac{\gamma}{2} |z^i|^2+\lambda \sum_{j\neq i} |z^j|^{1+\delta};
    $$
\item\label{A:H2} For $i=1,\cdots,n$, $g^i$ satisfies that $\as$, for each $(y,\bar y, z, \bar z)\in \R^n\times\R^n\times\R^{n\times d}\times\R^{n\times d}$,
    $$
    \begin{array}{ll}
    &\Dis |g^i(\omega,t,y,z)-g^i(\omega,t,\bar y,\bar z)| \vspace{0.1cm}\\
    &\Dis \leq \phi(|y|\vee|\bar y|) \Dis \left[ \left(1+|z|+|\bar z|\right)\left(|y-\bar y|+|z^i-\bar z^i|\right)+\left(1+|z|^\delta+|\bar z|^\delta\right)\sum_{j\neq i} |z^j-\bar z^j|\right];
    \end{array}
    $$

\item\label{A:H3} There exists two non-negative constants $C_1$ and $C_2$ such that
$$\|\xi\|_{\infty}\leq C_1\ \ \ \ {\rm and}\ \ \ \ \left\|\int_0^T \alpha_t{\rm d}t\right\|_{\infty}\leq C_2.$$

\end{enumerate}

In the first two assumptions~\ref{A:H1} and \ref{A:H2},  it creates no essential difference to replace  both terms $\sum_{j\neq i} |z^j|^{1+\delta}$ and $\sum_{j\neq i} |z^j-\bar z^j|$ with $|z|^{1+\delta}$ and $|z-\bar z|$,  respectively. The underlying way of formulation is more convenient for subsequent exposition.

\begin{thm}\label{thm:2.1}
Let assumptions \ref{A:H1}-\ref{A:H3} hold. Then, there exist a real $\eps>0$ (depending only on constants $(n,\gamma,\lambda,\delta,C_1,C_2)$ and function $\phi(\cdot)$) and a bounded subset $\mathcal{B}_\eps$ of the product space $\s^\infty_{[T-\eps,T]}(\R^n)\times {\rm BMO}_{[T-\eps,T]}(\R^{n\times d})$ such that BSDE \eqref{eq:1.1} has a unique local solution $(Y,Z)$ on the time interval $[T-\eps,T]$ with $(Y,Z)\in \mathcal{B}_\eps$.
\end{thm}

\begin{rmk}\label{rmk:2.2}
Assumptions (H1) and (H2) of \cref{thm:2.1} are more general than those of \citet[Theorem 2.2, p. 1072]{HuTang2016SPA} in that the former relaxes the growth and continuity of the generator in the first unknown variable $y$. For example,  the following generator $g$ satisfies the former,  while not the latter: $$g^i(\omega,t,y,z)=(|y|^2+\sin|z^i|)|z|+|z|^{3\over 2}+|z^i|^2,\ \ i=1,\cdots, n.$$
\end{rmk}

The second and third main results of this paper concern  global solutions of quadratic BSDEs with  bounded terminal values. The following two assumptions are further required.

\begin{enumerate}
\renewcommand{\theenumi}{(H4)}
\renewcommand{\labelenumi}{\theenumi}
\item\label{A:H4} For $i=1,\cdots,n$, $g^i$ satisfies that $\as$, for each $(y,z)\in \R^n\times\R^{n\times d}$,
    $$
    {\rm sgn}(y^i)g^i(\omega,t,y,z)\leq \alpha_t(\omega)+\beta |y|+\lambda |z|^{1+\delta}+\frac{\gamma}{2} |z^i|^2;
    $$

\renewcommand{\theenumi}{(H5)}
\renewcommand{\labelenumi}{\theenumi}
\item\label{A:H5} For $i=1,\cdots,n$, it holds that $\as$, for each $(y,z)\in \R^n\times\R^{n\times d}$,
    \begin{equation}\label{eq:2.1}
    g^i(\omega,t,y,z)\geq  \frac{\bar \gamma}{2} |z^i|^2-\alpha_t(\omega)-\beta |y|-\lambda |z|^{1+\delta}
    \end{equation}
    or
    \begin{equation}\label{eq:2.2}
    g^i(\omega,t,y,z)\leq  -\frac{\bar \gamma}{2} |z^i|^2+\alpha_t(\omega)+\beta |y|+\lambda |z|^{1+\delta}.\vspace{0.2cm}
    \end{equation}
\end{enumerate}

\begin{rmk}\label{rmk:2.3} Assumption \ref{A:H5} holds for the generator $g$ if some components of $g$ satisfy \eqref{eq:2.1}, and the others satisfy \eqref{eq:2.2}.
\end{rmk}

\begin{thm}\label{thm:2.4}
Let assumptions \ref{A:H1}-\ref{A:H4} be satisfied. If the constant $\lambda$ in \ref{A:H4} vanishes, then BSDE \eqref{eq:1.1} admits a unique global solution $(Y,Z)\in \s^\infty(\R^n)\times {\rm BMO}(\R^{n\times d})$ on $\T$.
\end{thm}

\begin{thm}\label{thm:2.5}
Let assumptions \ref{A:H1}-\ref{A:H5} hold. Then BSDE \eqref{eq:1.1} admits a unique global solution $(Y,Z)\in \s^\infty(\R^n)\times {\rm BMO}(\R^{n\times d})$ on $\T$.
\end{thm}


\begin{rmk}\label{rmk:2.6}
Assumption \ref{A:H4} is some kind of one-sided linear growth condition of the generator $g$ with respect to the variable $y$, and assumption \ref{A:H5} can be regarded as some kind of strictly quadratic condition of $g^i$ with respect to $z^i$. A generator $g$ satisfying assumptions \ref{A:H1}-\ref{A:H5} can still have a general growth in the variable $y$. For example,  the following generator $g$ satisfies all these assumptions:
$$g^i(\omega,t,y,z)=({\rm e}^{-y^i}+\cos|z^i|)|z|-|z|^{4\over 3}+(-1)^i|z^i|^2,\ \ i=1,\cdots, n. \vspace{-0.1cm}$$
Note that this $g$ does not satisfy the corresponding assumptions used
in \citet{HuTang2016SPA}.
\end{rmk}

For the sake of studying global solutions of multidimensional diagonally quadratic BSDEs with unbounded terminal values, we introduce the following four assumptions on the data $(g, \xi)$ of BSDEs.

\begin{enumerate}
\renewcommand{\theenumi}{(B\arabic{enumi})}
\renewcommand{\labelenumi}{\theenumi}
\item\label{A:B1} For $i=1,\cdots,n$, $g^i(\omega,t, y, z)$ varies with $(\omega, t, y)$ and the $i$th row $z^i$ of the  matrix $z\in \R^{n\times d}$ only, and  grows linearly in $y$ and quadratically in $z^i$, i.e., $\as$,
    $$
    |g^i(\omega,t,y,z)|\leq \alpha_t(\omega)+\beta |y|+\frac{\gamma}{2}|z|^2 \quad \hbox{for each $(y, z)\in \R^n\times\R^{1\times d}$};\vspace{-0.1cm}
    $$

\item\label{A:B2} $g$ is uniformly Lipschitz continuous in $y$, i.e., $\as$,
    $$
    |g(\omega,t,y,z)-g(\omega,t,\bar y, z)|\leq \beta |y-\bar y| \quad \hbox{ for each $(y, \bar y, z)\in (\R^n)^2\times\R^{1\times d}$};\vspace{-0.1cm}
    $$

\item\label{A:B3} $\as$, for each $i=1,\cdots,n$ and $y\in \R^n$, $g^i(\omega,t,y,\cdot)$ is either convex or concave;

\item\label{A:B4}  The terminal value $\xi$ is of exponential moments of arbitrary order as well as $\int_0^T \alpha_t {\rm d}t$. That is, we have for each $p\geq 1$,
$$
\E\left[\exp\left\{p\left(|\xi|+\int_0^T \alpha_t {\rm d}t\right)\right\}\right]<+\infty.\vspace{0.2cm}
$$
\end{enumerate}

\begin{rmk}\label{rmk:2.7}
Assumption \ref{A:B3} holds for the generator $g$ if some components of $g$ are convex in $z$, and the others are concave in $z$.
\end{rmk}

The following theorem seems to be the first result on the general solvability of systems of quadratic BSDEs with unbounded terminal values, and  constitutes the last main result of the  paper.

\begin{thm}\label{thm:2.8}
Let assumptions \ref{A:B1}-\ref{A:B4} be in force. Then BSDE \eqref{eq:1.1} admits a unique global solution $(Y,Z)\in \mathcal{E}(\R^n)\times \mcal(\R^{n\times d})$ on $\T$.
\end{thm}

\begin{rmk}\label{rmk:2.9}
 In the case of unbounded terminal values,   the martingale part of the first known process $Y$ goes beyond the space of BMO martingales, and  some delicate and technical computations are developed in the proof of \cref{thm:2.8}, in which a priori estimates on one-dimensional quadratic BSDEs, the $\theta$-method for convex functions and Doob's maximal inequality for martingales play a  crucial role.
\end{rmk}

\section{Local solution with bounded terminal value: proof of \cref{thm:2.1}}
\label{sec:3-local solution}
\setcounter{equation}{0}

For each $i=1,\cdots,n$, $H\in \R^{n\times d}$ and $z\in \R^{1\times d}$, define by $H(z;i)$ the matrix in $\R^{n\times d}$ whose $i$th row is $z$ and whose $j$th row is $H^j$ for any $j\neq i$.

Let assumptions \ref{A:H1}-\ref{A:H3} hold. For a pair of processes $(U,V)\in \s^\infty(\R^n)\times {\rm BMO}(\R^{n\times d})$,  we consider the following decoupled system of quadratic BSDEs:
\begin{equation}\label{eq:3.1}
  Y_t^i=\xi^i+\int_t^T g^i(s,U_s,V_s(Z_s^i;i)){\rm d}s-\int_t^T Z_s^i {\rm d}B_s, \ \ t\in\T;\  \ i=1,\cdots, n.
\end{equation}
For each fixed $i=1,\cdots,n$, in view of assumptions \ref{A:H1} and \ref{A:H2}, it is not difficult to verify that $\as$, for each $(z, \bar z)\in (\R^{1\times d})^2$,
$$
|g^i(t,U_t,V_t(z;i))|\leq \alpha_t+\phi(|U_t|)+n\lambda |V_t|^{1+\delta}+\frac{\gamma}{2} |z|^2
$$
and
$$
|g^i(t,U_t,V_t(z;i))-g^i(t,U_t,V_t(\bar z;i))| \leq \phi(|U_t|)\left(1+2|V_t|+|z|+|\bar z|\right) |z-\bar z|.\vspace{0.1cm}
$$
This means that the generator $g^i(t,U_t,V_t(z;i))$ satisfies assumptions \ref{A:A1} and \ref{A:A2} defined in Appendix. Then, in view of assumption \ref{A:H3}, it follows from \cref{lem:A.1} that for each $i=1,\cdots,n$, one-dimensional BSDE with the terminal value $\xi^i$ and the generator $g^i(t,U_t,V_t(z;i))$ has a unique solution $(Y^i,Z^i)$ such that $Y^i$ is (essentially) bounded and $Z^i\cdot B:=\left(\int_0^t Z_s^i{\rm d}B_s\right)_{t\in\T}$ is a BMO martingale. That is to say, the system of BSDEs
\eqref{eq:3.1} has a unique solution $(Y,Z)\in \s^\infty(\R^n)\times {\rm BMO}(\R^{n\times d})$.\vspace{0.2cm}

Now, define the solution map $\Gamma: (U,V)\mapsto \Gamma(U,V)$ as follows:
$$
\Gamma(U,V):=(Y,Z),\ \ \ \RE\ (U,V)\in \s^\infty(\R^n)\times {\rm BMO}(\R^{n\times d}).
$$
It is a transformation in the space $\s^\infty(\R^n)\times {\rm BMO}(\R^{n\times d})$. Moreover, it follows from \cref{lem:A.1} that for each $i=1,\cdots,n$, $t\in\T$ and stopping time $\tau$ with values in $[t,T]$,
$$
\begin{array}{lll}
|Y_t^i|&\leq & \Dis {1\over \gamma}\ln 2+\|\xi^i\|_{\infty}+\left\|\int_0^T \alpha_s{\rm d}s\right\|_{\infty}\vspace{0.1cm}\\
&&\Dis +\phi\left(\|U\|_{\s^{\infty}_{[t,T]}}\right)(T-t)
+\gamma^{\frac{1+\delta}{1-\delta}}C_{\delta,\lambda,n} \|V\|_{{\rm BMO}_{[t,T]}}^{2\frac{1+\delta}{1-\delta}}(T-t)
\end{array}
$$
and
$$
\begin{array}{lll}
&& \Dis\E_\tau\left[\int_\tau^T |Z_s^i|^2{\rm d}s\right]\vspace{0.1cm}\\
&\leq &\Dis {1\over \gamma^2} \exp(2\gamma \|\xi^i\|_\infty)+{1\over \gamma}\exp\left(2\gamma \left\|\sup_{s\in [t,T]}|Y_s^i|\right\|_\infty\right)\vspace{0.2cm}\\
&& \cdot \Dis\left(1+2\left\|\int_0^T \alpha_s {\rm d}s\right\|_{\infty}+2\phi\left(\|U\|_{\s^{\infty}_{[t,T]}}\right)(T-t)
+2C_{\delta,\lambda,n}\|V\|_{{\rm BMO}_{[t,T]}}^{2\frac{1+\delta}{1-\delta}}(T-t)\right),\vspace{0.1cm}
\end{array}
$$
where the constant $C_{\delta,\lambda,n}$ is defined in \eqref{eq:A.4} of Appendix. Therefore, in view of assumption \ref{A:H3}, for each $t\in\T$, we have
\begin{equation}\label{eq:3.2}
\begin{array}{lll}
\|Y\|_{\s^{\infty}_{[t,T]}} &\leq & \Dis {n\over \gamma}\ln 2+n(C_1+C_2)\vspace{0.1cm}\\
&&\Dis +n\phi\left(\|U\|_{\s^{\infty}_{[t,T]}}\right)(T-t)
+n\gamma^{\frac{1+\delta}{1-\delta}}C_{\delta,\lambda,n} \|V\|_{{\rm BMO}_{[t,T]}}^{2\frac{1+\delta}{1-\delta}}(T-t)
\end{array}
\end{equation}
and
\begin{equation}\label{eq:3.3}
\begin{array}{lll}
\Dis\|Z\|_{{\rm BMO}_{[t,T]}}^2 &\leq &\Dis {n\over \gamma^2} \exp(2\gamma C_1)+{n\over \gamma}\exp\left(2\gamma \|Y\|_{\s^{\infty}_{[t,T]}}\right)\vspace{0.2cm}\\
&& \cdot \Dis\left(1+2C_2+2\phi\left(\|U\|_{\s^{\infty}_{[t,T]}}\right)(T-t)
+2C_{\delta,\lambda,n}\|V\|_{{\rm BMO}_{[t,T]}}^{2\frac{1+\delta}{1-\delta}}(T-t)\right).
\end{array}
\end{equation}
Define
$$
K_1:={n\over \gamma}\ln 2+n(C_1+C_2),
$$
$$
K_2:={n\over \gamma^2} \exp(2\gamma C_1)+{n\over \gamma}\exp\left(4\gamma K_1\right)(1+2C_2)
$$
and
$$
\eps_0:=\left(\frac{K_1}{n\phi\left(2K_1\right)
+n\gamma^{\frac{1+\delta}{1-\delta}}C_{\delta,\lambda,n} (2K_2)^{\frac{1+\delta}{1-\delta}}}\right) \bigwedge \left(\frac{{\gamma\over n}\exp\left(-4\gamma K_1\right)K_2}{2\phi\left(2K_1\right)
+2C_{\delta,\lambda,n} (2K_2)^{\frac{1+\delta}{1-\delta}}}\right)>0.\vspace{0.3cm}
$$
By virtue of \eqref{eq:3.2} and \eqref{eq:3.3}, we can verify directly that for each $\eps\in (0,\eps_0]$, if
$$\|U\|_{\s^{\infty}_{[T-\eps,T]}}\leq 2K_1\ \ \ {\rm and}\ \ \  \|V\|_{{\rm BMO}_{[T-\eps,T]}}^2\leq 2K_2,$$
then
$$\|Y\|_{\s^{\infty}_{[T-\eps,T]}}\leq 2K_1\ \ \ {\rm and}\ \ \  \|Z\|_{{\rm BMO}_{[T-\eps,T]}}^2\leq 2K_2.$$
This means that
$$
\Gamma(U,V)\in \mathcal{B}_\eps,\ \ \RE\ (U,V)\in \mathcal{B}_\eps,
$$
where
\begin{equation}\label{eq:3.4}
\begin{array}{l}
\Dis\mathcal{B}_\eps:=\left\{(U,V)\in \s^\infty(\R^n)\times {\rm BMO}(\R^{n\times d}):\right.\vspace{0.1cm}\\
\Dis \hspace{1.5cm} \left.\|U\|_{\s^{\infty}_{[T-\eps,T]}}\leq 2K_1\ \ {\rm and}\ \   \|V\|_{{\rm BMO}_{[T-\eps,T]}}^2\leq 2K_2\right\}\vspace{0.2cm}
\end{array}
\end{equation}
is a Banach space with the following norm
$$
\|(U,V)\|_{\mathcal{B}_\eps}:=\sqrt{\|U\|_{\s^{\infty}_{[T-\eps,T]}}^2+\|V\|_{{\rm BMO}_{[T-\eps,T]}}^2},\ \ \ \RE\ (U,V)\in \mathcal{B}_\eps.
$$
That is, the mapping $\Gamma$ is stable in the Banach space $\mathcal{B}_\eps$ for each $\eps\in (0,\eps_0]$.

It remains to show that there exists a real $\eps\in (0,\eps_0]$ depending only on constants $(n,\gamma,\lambda,\delta,C_1,C_2)$ and function $\phi(\cdot)$ such that $\Gamma$ is a contraction in $\mathcal{B}_\eps$. Indeed, for any fixed $\eps\in (0,\eps_0]$, and $(U,V)\in \mathcal{B}_{\eps}$ and $(\widetilde U,\widetilde V)\in \mathcal{B}_{\eps}$, we set
$$
(Y,Z):=\Gamma(U,V),\ \ \ \ (\widetilde Y,\widetilde Z):=\Gamma(\widetilde U,\widetilde V).
$$
That is, for $i=1,\cdots,n$ and $t\in\T$,
$$
\begin{array}{l}
\Dis Y_t^i=\xi^i+\int_t^T g^i(s,U_s,V_s(Z_s^i;i)){\rm d}s-\int_t^T Z_s^i {\rm d}B_s, \vspace{0.2cm}\\
\Dis \widetilde Y_t^i=\xi^i+\int_t^T g^i(s,\widetilde U_s,\widetilde V_s(\widetilde Z_s^i;i)){\rm d}s-\int_t^T \widetilde Z_s^i {\rm d}B_s.
\end{array}
$$
Define for $i=1,\cdots,n$ and $s\in\T$,
$$
\Delta_s^{1,i}:=g^i(s,U_s,V_s(Z_s^i;i))-g^i(s,U_s,V_s(\widetilde Z_s^i;i)), \ \  \Delta_s^{2,i}:=g^i(s,U_s,V_s(\widetilde Z_s^i;i))-g^i(s,\widetilde U_s,\widetilde V_s(\widetilde Z_s^i;i)).
$$
Then,
\begin{equation}\label{eq:3.5}
\begin{array}{ll}
 \Dis \Dis Y_t^i-\widetilde Y_t^i+\int_t^T \left(Z_s^i-\widetilde Z_s^i \right) {\rm d}B_s-\int_t^T \Delta_s^{1,i}{\rm d}s
=\Dis \int_t^T\Delta_s^{2,i} {\rm d}s,\ \ \ \ t\in\T.
\end{array}
\end{equation}
It follows from assumption \ref{A:H2} that $\ass$, for each $i=1,\cdots,n$,
\begin{equation}\label{eq:3.6}
|\Delta_s^{1,i}|\leq \phi(|U_s|)\left(1+2|V_s|+|Z_s|+|\widetilde Z_s|\right)|Z_s^i-\widetilde Z_s^i|
\end{equation}
and
\begin{equation}\label{eq:3.7}
\begin{array}{lll}
\Dis |\Delta_s^{2,i}|&\leq &\Dis \phi(|U_s|\vee |\widetilde U_s|)\left[\left(1+|V_s|+|\widetilde V_s|+|Z_s|+|\widetilde Z_s|\right)|U_s-\widetilde U_s|
\right.\vspace{0.1cm}\\
&& \Dis \hspace{2.6cm}+\left.\sqrt{n}\left(1+|V_s|^{\delta}+|\widetilde V_s|^{\delta}+2|\widetilde Z_s|^{\delta}\right)|V_s-\widetilde V_s|\right].\vspace{0.1cm}
\end{array}
\end{equation}
For $i=1,\cdots,n$, by \eqref{eq:3.6} we can define the $\R^d$-valued process $G(i)$ in an obvious way such that
\begin{equation}\label{eq:3.8}
\Delta_s^{1,i}=\left(Z_s^i-\widetilde Z_s^i \right)G_s(i)\ \ {\rm and}\ \  |G_s(i)|\leq \phi(|U_s|)\left(1+2|V_s|+|Z_s|+|\widetilde Z_s|\right),\ \ \ s\in\T.
\end{equation}
Finally, in view of the fact that all pairs of processes $(U,V), (\widetilde U,\widetilde V)$ and $(Y,Z), (\widetilde Y,\widetilde Z)$ are in $\mathcal{B}_{\eps}$ together with definition \eqref{eq:3.4} of $\mathcal{B}_{\eps}$ and inequalities \eqref{eq:3.5}-\eqref{eq:3.8}, using Girsanov's transform, H\"{o}lder's inequality, the energy inequality for BMO martingale and Lemma A.4 in \citet{HuTang2016SPA} we can follow the argument in pages 1078-1079 of \citet{HuTang2016SPA} to get the desired conclusion. \cref{thm:2.1} is then proved.

\section{Global solution with bounded terminal value: proof of \cref{thm:2.4,thm:2.5}}
\label{sec:4-global solution}
\setcounter{equation}{0}

We need the following lemma.

\begin{lem}\label{lem:4.1}
Let assumption \ref{A:H3} hold. Assume that for some $h\in (0,T]$, the BSDE \eqref{eq:1.1} has a solution $(Y,Z)\in \s^{\infty}_{[T-h,T]}\times {\rm BMO}_{[T-h,T]}$ on time interval $[T-h,T]$. We have

(i) If the generator $g$ satisfies assumption \ref{A:H4} with $\lambda=0$, then
$$
\|Y\|_{\s^{\infty}_{[T-h,T]}}\leq (2n)^{[2n\beta T]+2} C_1+\left(2n+(2n)^2+\cdots+(2n)^{[2n\beta T]+2}\right)C_2,
$$
where and hereafter $[x]$ denotes the maximum of integers smaller than or equal to $x$.

(ii) If the generator $g$ satisfies assumptions \ref{A:H4} and \ref{A:H5}, then
$$
\|Y\|_{\s^{\infty}_{[T-h,T]}}\leq (4n)^{[4n\beta T]+2} C_1+\left(4n+(4n)^2+\cdots+(4n)^{[4n\beta T]+2}\right)C_5,
$$
where $C_5$ is a positive constant depending only on $(n,\beta,\gamma,\bar\gamma,\lambda,\delta,T,C_2)$.
\end{lem}

\begin{proof}
For $i=1,\cdots,n$, it follows from assumption \ref{A:H4} that
$$
{\rm sgn} (Y_t^i(\omega))g^i\left(\omega,t,Y_t(\omega),Z_t(\omega)\right)\leq \alpha_t(\omega)+\beta |Y_t(\omega)|+\lambda |Z_t(\omega)|^{1+\delta}+\frac{\gamma}{2} |Z_t^i(\omega)|^2,\ \ t\in [T-h,T],
$$
which means that the generator of $i$th equation in system of BSDEs \eqref{eq:1.1} satisfies \eqref{eq:A.8} in Appendix (by letting $f(\omega,t,z)\equiv g^i\left(\omega,t,Y_t(\omega),Z_t(\omega)\right)$). It then follows from (i) of \cref{lem:A.2} and assumption \ref{A:H3} that for each $i=1,\cdots,n$,
\begin{equation}\label{eq:4.1}
\begin{array}{lll}
\Dis \exp\left(\gamma |Y_t^i|\right)&\leq &\Dis \exp\left(\gamma(C_1+C_2 )+\beta\gamma \|Y\|_{\s^{\infty}_{[t,T]}}(T-t)\right)\vspace{0.2cm}\\
&&\Dis\  \cdot\E_t\left[\exp\left(\lambda\gamma \int_t^T |Z_s|^{1+\delta} {\rm d}s\right)\right],\ \ \ \ t\in [T-h,T].\vspace{0.2cm}
\end{array}
\end{equation}

(i) In the case of $\lambda=0$, it follows from \eqref{eq:4.1} that for $i=1,\cdots,n$,
$$
|Y_t^i|\leq C_1+C_2+\beta \|Y\|_{\s^{\infty}_{[t,T]}}(T-t),\ \ t\in [T-h,T].
$$
Therefore,
\begin{equation}\label{eq:4.2}
\|Y\|_{\s^{\infty}_{[t,T]}}\leq n(C_1+C_2)+n\beta \|Y\|_{\s^{\infty}_{[t,T]}}(T-t),\ \ t\in [T-h,T].\vspace{0.2cm}
\end{equation}

For $\beta=0$, it is clear that
\begin{equation}\label{eq:4.3}
\|Y\|_{\s^{\infty}_{[T-h,T]}}\leq n(C_1+C_2).
\end{equation}
Otherwise, let $m_0$ is the unique positive integer satisfying
$$
T-h \in [T-m_0\eps,T-(m_0-1)\eps)
$$
or, equivalently,
\begin{equation}\label{eq:4.4}
2n\beta h={h\over \eps}\leq m_0 <{h\over \eps}+1\leq 2n\beta T+1
\end{equation}
with
$$
\eps:=\frac{1}{2n\beta}>0.\vspace{0.1cm}
$$
If $m_0=1$, then $n\beta(T-t)\leq n\beta h\leq n\beta\eps=1/2$ for $t\in [T-h,T]$, and it follows from \eqref{eq:4.2} that
$$
\|Y\|_{\s^{\infty}_{[t,T]}}\leq 2n(C_1+C_2),\ \ t\in [T-h,T],
$$
which yields that
\begin{equation}\label{eq:4.5}
\|Y\|_{\s^{\infty}_{[T-h,T]}}\leq 2n(C_1+C_2).\vspace{0.1cm}
\end{equation}
If $m_0=2$, then $n\beta(T-t)\leq n\beta\eps=1/2$ for $t\in [T-\eps,T]$, and it follows from \eqref{eq:4.2} that
$$
\|Y\|_{\s^{\infty}_{[t,T]}}\leq 2n(C_1+C_2),\ \ t\in [T-\eps,T],
$$
which yields that
\begin{equation}\label{eq:4.6}
\|Y_{T-\eps}\|_{\infty}\leq \|Y\|_{\s^{\infty}_{[T-\eps,T]}}\leq 2n(C_1+C_2).
\end{equation}
Now, consider the following system of BSDEs
$$
  Y_t=Y_{T-\eps}+\int_t^{T-\eps} g(s,Y_s,Z_s){\rm d}s-\int_t^{T-\eps} Z_s {\rm d}B_s, \ \ t\in [T-h,T-\eps].
$$
In view of \eqref{eq:4.6}, identically as in obtaining \eqref{eq:4.5} , we have
$$
\|Y\|_{\s^{\infty}_{[T-h,T-\eps]}}\leq 2n\left(2n(C_1+C_2)+C_2\right),
$$
and therefore,
$$
\|Y\|_{\s^{\infty}_{[T-h,T]}}\leq (2n)^2 C_1+\left(2n+(2n)^2\right)C_2.\vspace{0.1cm}
$$
Proceeding the above computation gives that if $m_0$ satisfies \eqref{eq:4.4}, then
$$
\|Y\|_{\s^{\infty}_{[T-h,T]}}\leq (2n)^{m_0} C_1+\left(2n+(2n)^2+\cdots+(2n)^{m_0}\right)C_2,
$$
which together with \eqref{eq:4.4} and \eqref{eq:4.3} yields assertion (i)  immediately.\vspace{0.2cm}

(ii) For $i=1,\cdots,n$, it follows from assumption \ref{A:H5} that
$$
g^i\left(\omega,t,Y_t(\omega),Z_t(\omega)\right)\geq \frac{\bar\gamma}{2} |Z_t^i(\omega)|^2-\alpha_t(\omega)-\beta |Y_t(\omega)|-\lambda |Z_t(\omega)|^{1+\delta}\vspace{-0.1cm}
$$
or\vspace{-0.1cm}
$$
g^i\left(\omega,t,Y_t(\omega),Z_t(\omega)\right)\leq -\frac{\bar\gamma}{2} |Z_t^i(\omega)|^2+\alpha_t(\omega)+\beta |Y_t(\omega)|+\lambda |Z_t(\omega)|^{1+\delta},
$$
which means that the generator of $i$th equation in system of BSDEs \eqref{eq:1.1} satisfies \eqref{eq:A.9} or \eqref{eq:A.10} in Appendix (by letting $f(\omega,t,z)\equiv g^i\left(\omega,t,Y_t(\omega),Z_t(\omega)\right)$). It then follows from (ii) of \cref{lem:A.2} and assumption \ref{A:H3} that for each $i=1,\cdots,n$,
$$
\begin{array}{ll}
    &\Dis \E_t\left[\exp\left(\frac{\bar \gamma}{2}\eps_0 \int_t^T |Z_s^i|^2{\rm d}s \right)\right]\vspace{0.2cm}\\
    \leq & \Dis\E_t\left[\exp\left(6\eps_0 \left\|\sup_{s\in [t,T]}|Y_s^i|\right\|_{\infty}+3\eps_0 \left\|\int_0^T \alpha_s {\rm d}s\right\|_{\infty}+3\eps_0 \beta T \|Y\|_{\s^{\infty}_{[t,T]}}+3\eps_0 \lambda \int_t^T |Z_s|^{1+\delta} {\rm d}s\right)\right]
    \vspace{0.2cm}\\
    \leq &\Dis \exp\left(3\eps_0  C_2+3\eps_0 (2+\beta T)\|Y\|_{\s^{\infty}_{[t,T]}} \right) \E_t\left[\exp\left(3\eps_0 \lambda \int_t^T |Z_s|^{1+\delta} {\rm d}s\right)\right],\ \ \ t\in [T-h,T],
\end{array}
$$
where
$$
\eps_0:=\left(\frac{\bar\gamma}{9}\right)\bigwedge \left(\frac{\gamma}{12(\beta T+2)}\right)>0.\vspace{0.1cm}
$$
Thus, by H\"{o}lder's inequality we get that for $t\in [T-h,T]$,\vspace{0.1cm}
\begin{equation}\label{eq:4.7}
\begin{array}{ll}
    &\Dis \E_t\left[\exp\left(\frac{\bar \gamma\eps_0}{2n} \int_t^T |Z_s|^2{\rm d}s \right)\right]\vspace{0.2cm}\\
    \leq &\Dis \exp\left(3\eps_0  C_2+3\eps_0 (2+\beta T)\|Y\|_{\s^{\infty}_{[t,T]}} \right) \E_t\left[\exp\left(3\eps_0 \lambda \int_t^T |Z_s|^{1+\delta} {\rm d}s\right)\right].
\end{array}
\end{equation}
Note by Young's inequality that for each pair of $a,b>0$,
\begin{equation}\label{eq:4.8}
ab^{1+\delta}=\left(\left(\frac{1+\delta}{2}\right)^{\frac{1+\delta}{1-\delta}} a^{\frac{2}{1-\delta}}\right)^{\frac{1-\delta}{2}}
\left(\frac{2}{1+\delta}b^2\right)^{\frac{1+\delta}{2}}\leq b^2+\frac{1-\delta}{2}\left(\frac{1+\delta}{2}\right)^{\frac{1+\delta}{1-\delta}} a^{\frac{2}{1-\delta}}.
\end{equation}
By letting $a=12n\lambda/\bar\gamma$ and $b=|Z_s|$ in \eqref{eq:4.8}, we have
\begin{equation}\label{eq:4.9}
3\eps_0\lambda |Z_s|^{1+\delta}=\frac{\bar\gamma\eps_0}{4n}\left(\frac{12n\lambda}{\bar\gamma} |Z_s|^{1+\delta}\right)\leq \frac{\bar\gamma\eps_0}{4n} |Z_s|^2+ C_3,\ \ \ s\in\T,
\end{equation}
where
$$
C_3:=\frac{\bar\gamma\eps_0(1-\delta)}{8n}\left(\frac{1+\delta}{2}\right)^{\frac{1+\delta}{1-\delta}} \left(\frac{12n\lambda}{\bar\gamma}\right)^{\frac{2}{1-\delta}}.\vspace{0.2cm}
$$
Coming back to \eqref{eq:4.7}, by \eqref{eq:4.9} and H\"{o}lder's inequality we deduce that for $t\in [T-h,T]$,
\begin{equation}\label{eq:4.10}
\E_t\left[\exp\left(\frac{\bar \gamma\eps_0}{2n} \int_t^T |Z_s|^2{\rm d}s \right)\right]\leq \exp\left(6\eps_0 C_2+ 2C_3 T +6\eps_0 (2+\beta T)\|Y\|_{\s^{\infty}_{[t,T]}} \right).\vspace{0.2cm}
\end{equation}

On the other hand, it follows from \eqref{eq:4.1} and Jensen's inequality  that
\begin{equation}\label{eq:4.11}
\begin{array}{lll}
\Dis \exp\left(\gamma |Y_t|\right)&\leq &\Dis \exp\left(n\gamma(C_1+C_2 )+n\beta\gamma \|Y\|_{\s^{\infty}_{[t,T]}}(T-t)\right)\vspace{0.1cm}\\
&&\Dis \ \cdot\E_t\left[\exp\left(n\lambda\gamma \int_t^T |Z_s|^{1+\delta} {\rm d}s\right)\right],\ \ \ \ t\in [T-h,T].
\end{array}
\end{equation}
By letting $a=2 n^2\lambda\gamma/\bar\gamma\eps_0$ and $b=|Z_s|$ in \eqref{eq:4.8}, we have
\begin{equation}\label{eq:4.12}
n\lambda \gamma |Z_s|^{1+\delta}=\frac{\bar\gamma\eps_0}{2n}\left(\frac{2 n^2\lambda\gamma}{\bar\gamma\eps_0} |Z_s|^{1+\delta}\right)\leq \frac{\bar\gamma\eps_0}{2n} |Z_s|^2+ C_4,\ \ \ s\in\T,
\end{equation}
where
$$
C_4:=\frac{\bar\gamma\eps_0(1-\delta)}{4n}\left(\frac{1+\delta}{2}
\right)^{\frac{1+\delta}{1-\delta}} \left(\frac{2 n^2\lambda\gamma}{\bar\gamma\eps_0}\right)^{\frac{2}{1-\delta}}.\vspace{0.2cm}
$$
Combining \eqref{eq:4.10}-\eqref{eq:4.12} yields that
$$
\begin{array}{lll}
|Y_t|& \leq &\Dis n(C_1+C_2 )+{C_4 T\over \gamma}+\frac{6\eps_0 C_2+ 2C_3 T}{\gamma} +\frac{6\eps_0 (2+\beta T)}{\gamma}\|Y\|_{\s^{\infty}_{[t,T]}}\vspace{0.1cm}\\
&&\Dis +n\beta \|Y\|_{\s^{\infty}_{[t,T]}}(T-t)
,\ \ \ \ \ \ t\in [T-h,T].
\end{array}
$$
And, from the definition of $\eps_0$ it follows that
\begin{equation}\label{eq:4.13}
\|Y\|_{\s^{\infty}_{[t,T]}} \leq \Dis 2n(C_1+C_5)+2n\beta \|Y\|_{\s^{\infty}_{[t,T]}}(T-t),\ \ \ \ t\in [T-h,T],
\end{equation}
where
$$
C_5:=C_2+\frac{6\eps_0 C_2+ 2C_3 T}{n\gamma} +\frac{C_4 T}{n\gamma}.\vspace{0.1cm}
$$

Finally, observing that \eqref{eq:4.13} is almost the same as \eqref{eq:4.2}, we can use the same computation as in (i) to obtain the desired conclusion of (ii). The proof of \cref{lem:4.1} is then complete.
\end{proof}

\begin{rmk}\label{rmk:4.2}
Observe that the term $\|U\|_{\s^{\infty}_{[t,T]}}(T-t)$ in the conclusion (i) of \cref{lem:A.1} can be replaced with $\int_t^T \|U\|_{\s^{\infty}_{[s,T]}}{\rm d}s$. Then the term $\|Y\|_{\s^{\infty}_{[t,T]}}(T-t)$ in inequalities \eqref{eq:4.1} and \eqref{eq:4.2} can be replaced with $\int_t^T \|Y\|_{\s^{\infty}_{[s,T]}}{\rm d}s$.
Therefore, by Gronwall's inequality, we obtain the following better upper bound under the assumptions of (i) in \cref{lem:4.1}:
$$
\|Y\|_{\s^{\infty}_{[T-h,T]}}\leq n(C_1+C_2)\exp\left(n\beta h\right)\leq n(C_1+C_2)\exp\left(n\beta T\right).
$$
The same computation yields the following estimate under the assumptions of (ii) in \cref{lem:4.1}:
$$
\|Y\|_{\s^{\infty}_{[T-h,T]}}\leq 2n(C_1+C_5)\exp\left(2n\beta h\right)\leq 2n(C_1+C_5)\exp\left(2n\beta T\right).\vspace{0.2cm}
$$
\end{rmk}

Using \cref{thm:2.1} and \cref{lem:4.1}, we can follow the proof of  \citet[Theorem 4.1]{cheridito2015Stochastics} to derive our  \cref{thm:2.4,thm:2.5}. All the details are omitted here.

\section{Global solution with unbounded terminal value: proof of \cref{thm:2.8}}
\label{sec:5-global solution unbounde terminal value}
\setcounter{equation}{0}

Let assumptions \ref{A:B1}-\ref{A:B4} be in force. For a pair of processes $(U,V)\in \mathcal{E}(\R^n)\times \mcal (\R^{n\times d})$,  we consider the following decoupled system of quadratic BSDEs:
\begin{equation}\label{eq:5.1}
  Y_t^i=\xi^i+\int_t^T g^i(s,U_s,Z_s^i){\rm d}s-\int_t^T Z_s^i {\rm d}B_s, \ \ t\in\T;\  \ i=1,\cdots, n.
\end{equation}
For each fixed $i=1,\cdots,n$, in view of assumptions \ref{A:B1} and \ref{A:B3}, it is clear that $\as$,
$$
|g^i(t,U_t,z)|\leq \alpha_t+\beta |U_t|+\frac{\gamma}{2} |z|^2,\ \ \RE\ z\in\R^{1\times d},
$$
and $g^i(t,U_t,z)$ is convex or concave in $z$. Furthermore, in view of assumption \ref{A:B4} and the fact that $U\in \mathcal{E}(\R^n)$, by H\"{o}lder's inequality we have
$$
\RE\ q> 1,\ \  \E\left[\exp\left\{q\left(|\xi^i|+\int_0^T \left(\alpha_t+\beta |U_t|\right){\rm d}t\right)\right\}\right]<+\infty,\ \ \ \ i=1,\cdots,n.
$$
It then follows from Corollary 6 in \citet{BriandHu2008PTRF} that for each $i=1,\cdots,n$, the BSDE
$$
Y_t^i=\xi^i+\int_t^T g^i(s,U_s,Z_s^i){\rm d}s-\int_t^T Z_s^i {\rm d}B_s, \ \ t\in\T
$$
has a unique adapted solution $(Y^i,Z^i)$ such that for each $q>1$,
$$\E\left[\exp\left(q\sup_{t\in \T}|Y_t^i|\right)+\left(\int_0^T |Z_t^i|^2{\rm d}t\right)^{q\over 2}\right]<+\infty,$$
which means that, in view of H\"{o}lder's inequality, the system of BSDEs \eqref{eq:5.1} admits a unique solution $(Y,Z)$ in the space of processes $\mathcal{E}(\R^n)\times \mcal (\R^{n\times d})$.\vspace{0.2cm}

Based on the above argument, we can set $(Y^{(0)},Z^{(0)})=(0,0)$ and define, recursively, the sequence of processes $\{(Y^{(m)},Z^{(m)})\}_{m=1}^\infty$ in the space of processes $\mathcal{E}(\R^n)\times \mcal (\R^{n\times d})$ by the unique adapted solution of system of BSDEs:
\begin{equation}\label{eq:5.2}
  Y_t^{(m+1);i}=\xi^i+\int_t^T g^i(s,Y_s^{(m)},Z_s^{(m+1);i}){\rm d}s-\int_t^T Z_s^{(m+1);i}{\rm d}B_s, \ \ t\in\T;\  \ i=1,\cdots, n,
\end{equation}
where for sake of convenience, we denote by $Y^{(m);i}$ and $Z^{(m);i}$, respectively, the $i$th component of $Y^{(m)}$ and the $i$th row of $Z^{(m)}$. In the sequel, we will show that $\{(Y^{(m)},Z^{(m)})\}_{m=1}^\infty$ is a Cauchy sequence in the space $\s^q(\R^n)\times \hcal^q(\R^{n\times d})$ for each $q\geq 1$, and then converges to a pair of adapted processes $(Y,Z)$ in $\mathcal{E}(\R^n)\times \mcal(\R^{n\times d})$, which is the unique desired solution of system of BSDE \eqref{eq:1.1}.\vspace{0.2cm}

We first prove that
\begin{equation}\label{eq:5.3}
\RE\ q>1,\ \ \ \sup_{m\geq 0}\E\left[\exp\left(q\gamma \sup_{t\in \T}|Y_t^{(m)}|\right)\right]\leq K(q),
\end{equation}
where
$$
\begin{array}{l}
\Dis K(q):=\left(A(2q)A(8nq)\right)^{[2n\beta T]+1}\E\left[\exp\left(4n(8n)^{[2n\beta T]+1}q\gamma |\xi|\right)\right]\vspace{0.1cm}\\
\Dis \hspace*{1.5cm}\cdot \E\left[\exp\left(4n(16n)^{[2n\beta T]+1}q\gamma\int_0^T \alpha_s{\rm d}s \right)\right] <+\infty
\end{array}\vspace{0.1cm}
$$
with
\begin{equation}\label{eq:5.3*}
A(q):=\left({q\over q-1}\right)^{2q}.\vspace{0.1cm}
\end{equation}
In fact, for each $i=1,\cdots,n$ and $m\geq 0$, it follows from \ref{A:B1} and \ref{A:B3} that $\as$,
\begin{equation}\label{eq:5.4}
\RE\ z\in\R^{1\times d},\ \ \ |g^i(t,Y_t^{(m)},z)|\leq \alpha_t+\beta |Y_t^{(m)}|+\frac{\gamma}{2} |z|^2,
\end{equation}
and $g^i(t,Y_t^{(m)},z)$ is convex or concave in $z$. Furthermore, in view of assumption \ref{A:B4} and the fact that $Y^{(m)}\in \mathcal{E}(\R^n)$, by H\"{o}lder's inequality we have
$$
\RE\ p\geq 1,\ \  \E\left[\exp\left\{p\left(\sup_{t\in \T}|Y_t^{(m+1);i}|+\int_0^T \left(\alpha_t+\beta |Y_t^{(m)}|\right){\rm d}t\right)\right\}\right]<+\infty.
$$
Then, we can use \cref{lem:A.3} in Appendix to get that for $i=1,\cdots,n$ and $m\geq 0$,
$$
\exp\left(\gamma |Y_t^{(m+1);i}|\right)\leq \E_t\left[\exp\left(\gamma |\xi^i|+\gamma\int_t^T \left(\alpha_s+\beta |Y_s^{(m)}|\right) {\rm d}s\right)\right],\ \ t\in\T.
$$
Consequently, by Jensen's inequality we have, for each $m\geq 0$,
\begin{equation}\label{eq:5.5}
\exp\left(\gamma |Y_t^{(m+1)}|\right)\leq \E_t\left[\exp\left(n\gamma |\xi|+n\gamma\int_t^T \left(\alpha_s+\beta |Y_s^{(m)}|\right) {\rm d}s\right)\right],\ \ t\in\T.
\end{equation}
In view of \eqref{eq:5.5}, Doob's maximal inequality for martingales together with H\"{o}lder's inequality yields that for each $q>1$, $m\geq 0$ and $t\in\T$, we have
\begin{equation}\label{eq:5.6}
\begin{array}{ll}
&\Dis \E\left[\exp\left(q\gamma \sup_{s\in [t,T]}|Y_s^{(m+1)}|\right)\right]\vspace{0.2cm}\\
\leq &\Dis\left(\frac{q}{q-1}\right)^q \E\left[\exp\left(nq\gamma |\xi|+nq\gamma\int_t^T \alpha_s{\rm d}s+nq\gamma \beta \sup_{s\in [t,T]}|Y_s^{(m)}|(T-t) \right)\right]\vspace{0.2cm}\\
\leq &\Dis  [C(q)]^{1/2}\left\{\E\left[\exp\left(2nq\gamma \beta \sup_{s\in [t,T]}|Y_s^{(m)}| (T-t)\right)\right]\right\}^{1/2},
\end{array}
\end{equation}
where, in view of assumption \ref{A:B4}, 
$$
C(q):=A(q)\E\left[\exp\left(2nq\gamma |\xi|+2nq\gamma\int_0^T \alpha_s{\rm d}s \right)\right]<+\infty
$$
with $A(q)$ being defined in \eqref{eq:5.3*}.\vspace{0.2cm}

For $\beta=0$, it is clear from \eqref{eq:5.6} that
\begin{equation}\label{eq:5.7}
\RE\ q>1,\ \ \ \sup_{m\geq 0}\E\left[\exp\left(q\gamma \sup_{t\in [0,T]}|Y_t^{(m)}|\right)\right]\leq \sqrt{C(q)}.
\end{equation}
Otherwise, let $m_0$ is the unique positive integer satisfying
$T-m_0\eps\leq 0<T-(m_0-1)\eps$ or, equivalently,
\begin{equation}\label{eq:5.8}
2n\beta T={T\over \eps}\leq m_0 <2n\beta T+1
\end{equation}
with
$$
\eps:=\frac{1}{2n\beta}>0.
$$
If $m_0=1$, then $2nq\gamma\beta(T-t)\leq 2nq\gamma\beta T\leq 2nq\gamma\beta\eps=q\gamma$ for $t\in [0,T]$, and it follows from \eqref{eq:5.6} that for each $m\geq 0$, $q>1$ and $t\in\T$ ,
$$
\E\left[\exp\left(q\gamma \sup_{s\in [t,T]}|Y_s^{(m+1)}|\right)\right]\leq \sqrt{C(q)} \left(\E\left[\exp\left(q\gamma \sup_{s\in [t,T]}|Y_s^{(m)}|\right)\right]\right)^{1/2},
$$
and, by induction, \vspace{0.1cm}
$$
\E\left[\exp\left(q\gamma \sup_{s\in [t,T]}|Y_s^{(m+1)}|\right)\right]\leq \sqrt{C(q)}^{1+{1\over 2}+\cdots +{1\over 2^m}} \left(\E\left[\exp\left(q\gamma \sup_{s\in [t,T]}|Y_s^{(0)}|\right)\right]\right)^{1\over 2^{m+1}}\leq C(q).\vspace{0.2cm}
$$
Consequently, by H\"{o}lder's inequality we have, for each $q>1$,
\begin{equation}\label{eq:5.9}
\sup_{m\geq 0}\E\left[\exp\left(q\gamma \sup_{t\in \T}|Y_t^{(m)}|\right)\right]\leq A(q)\E\left[\exp\left(4nq\gamma |\xi|\right)\right]\E\left[\exp\left(4nq\gamma\int_0^T \alpha_s{\rm d}s \right)\right]<+\infty.
\end{equation}
If $m_0=2$, then $2nq\gamma\beta(T-t)\leq 2nq\gamma\beta \eps=q\gamma$ for $t\in [T-\eps,T]$, and from the argument in the case of $m_0=1$ it follows that for each $q>1$,
\begin{equation}\label{eq:5.10}
\sup_{m\geq 0}\E\left[\exp\left(q\gamma \sup_{t\in [T-\eps,T]}|Y_t^{(m)}|\right)\right]\leq A(q)\E\left[\exp\left(4nq\gamma |\xi|\right)\right]\E\left[\exp\left(4nq\gamma\int_0^T \alpha_s{\rm d}s \right)\right]<+\infty,
\end{equation}
which yields that
\begin{equation}\label{eq:5.11}
\sup_{m\geq 0}\E\left[\exp\left(q\gamma |Y_{T-\eps}^{(m)}|\right)\right]\leq A(q)\E\left[\exp\left(4nq\gamma |\xi|\right)\right]\E\left[\exp\left(4nq\gamma\int_0^T \alpha_s{\rm d}s \right)\right]<+\infty.
\end{equation}
Now, consider the following system of BSDEs: for $i=1,\cdots, n$,
$$
Y_t^{(m+1);i}=Y_{T-\eps}^{(m+1);i}+\int_t^{T-\eps} g^i(s,Y_s^{(m)},Z_s^{(m+1);i}){\rm d}s-\int_t^{T-\eps} Z_s^{(m+1);i}{\rm d}B_s, \ \ t\in [0,T-\eps].
$$
In view of \eqref{eq:5.11}, a similar argument as that obtaining \eqref{eq:5.9} yields that for each $q>1$,
$$
\begin{array}{lll}
&\Dis \sup_{m\geq 0}\E\left[\exp\left(q\gamma \sup_{t\in [0,T-\eps]}|Y_t^{(m)}|\right)\right]\vspace{0.1cm}\\
\leq &A(q)\Dis \sup_{m\geq 0}\E\left[\exp\left(4nq\gamma |Y_{T-\eps}^{(m)}|\right)\right]\E\left[\exp\left(4nq\gamma\int_0^T \alpha_s{\rm d}s \right)\right],\vspace{0.1cm}\\
\leq &\Dis A(q)A(4nq)\E\left[\exp\left(16n^2q\gamma |\xi|\right)\right]\E\left[\exp\left(32n^2q\gamma\int_0^T \alpha_s{\rm d}s \right)\right]<+\infty,
\end{array}
$$
and then, in view of \eqref{eq:5.10} and H\"{o}lder's inequality,
$$
\begin{array}{ll}
& \Dis \sup_{m\geq 0}\E\left[\exp\left(q\gamma \sup_{t\in \T}|Y_t^{(m)}|\right)\right]\\
\leq & \Dis A(2q)A(8nq)\E\left[\exp\left(32 n^2q\gamma |\xi|\right)\right]\E\left[\exp\left(64n^2q\gamma\int_0^T \alpha_s{\rm d}s \right)\right]<+\infty.
\end{array}
$$
Proceeding the above computation gives that if $m_0$ satisfies \eqref{eq:5.8}, then for each $q>1$,
$$
\begin{array}{ll}
&\Dis \sup_{m\geq 0}\E\left[\exp\left(q\gamma \sup_{t\in \T}|Y_t^{(m)}|\right)\right]\vspace{0.1cm}\\
\leq &\Dis \left(A(2q)A(8nq)\right)^{m_0-1}\E\left[\exp\left(4n(8n)^{m_0-1}q\gamma |\xi|\right)\right]\E\left[\exp\left(4n(16n)^{m_0-1}q\gamma\int_0^T \alpha_s{\rm d}s \right)\right]<+\infty,
\end{array}
$$
which together with \eqref{eq:5.8} and \eqref{eq:5.7} yields the desired conclusion \eqref{eq:5.3}.\vspace{0.2cm}

In the sequel, we show that
\begin{equation}\label{eq:5.12}
\RE\ q> 1,\ \ \ \sup_{m\geq 0}\E\left[\left(\int_0^T |Z_s^{(m)}|^2 {\rm d}s\right)^{q/2}\right]<+\infty.
\end{equation}
In fact, using It\^{o}-Tanaka's formula to compute $\exp(2\gamma |Y_t^{(m+1;i)}|)$ and utilizing inequality \eqref{eq:5.4}, we can deduce that for each $i=1,\cdots,n$ and $m\geq 0$,
$$
\begin{array}{ll}
&\Dis \gamma^2\int_0^T \exp(2\gamma |Y_s^{(m+1);i}|) |Z_s^{(m+1);i}|^2{\rm d}s\vspace{0.1cm}\\
\leq &\Dis \exp(2\gamma |\xi^i|)+2\gamma \int_0^T \exp(2\gamma |Y_s^{(m+1);i}|)\left(\alpha_s+\beta |Y_s^{m}|)\right){\rm d}s\vspace{0.1cm}\\
&\Dis -2\gamma \int_0^T \exp(2\gamma |Y_s^{(m+1);i}|){\rm sgn}(Y_s^{(m+1);i}) Z_s^{(m+1);i}{\rm d}B_s.
\end{array}
$$
In view of the previous inequality, we can use the BDG inequality and Young's inequality to derive that for each $q>1$, $m\geq 0$ and $i=1,\cdots,n$,
\begin{equation}\label{eq:5.13}
\begin{array}{lll}
&&\Dis \E\left[\left(\int_0^T |Z_s^{(m+1);i}|^2 {\rm d}s\right)^{q/2}\right]\vspace{0.1cm}\\
&\leq &\Dis  C \left(\E\left[\exp\left(2q\gamma \sup_{s\in\T}|Y_s^{(m+1);i}|\right)\right]+ \E\left[\left(\int_0^T \alpha_s +\beta |Y_s^{(m)}| {\rm d}s\right)^{q}\right] \right)\vspace{0.2cm}\\
&\leq &\Dis \bar C \left(\E\left[\exp\left(\int_0^T \alpha_s{\rm d}s\right)\right]+\sup_{m\geq 0} \E\left[\exp\left(2q\gamma \sup_{s\in\T}|Y_s^{(m)}|\right)\right]\right),\vspace{0.1cm}
\end{array}
\end{equation}
where $C$ and $\bar C$ are two positive constants depending only on $(q,\gamma)$ and $(q,\gamma,\beta,T)$,  respectively. Then the inequality \eqref{eq:5.12} follows immediately from \eqref{eq:5.13} and \eqref{eq:5.3}.\vspace{0.2cm}

Next, without  loss of generality, we assume that the generator $g$ is component-wisely  convex  in \ref{A:B3}, that is, $\as$, $g^i(\omega,t,y,\cdot)$ is convex for each $i=1,\cdots,n$ and $y\in \R^n$. Otherwise, if $g^i(\omega,t,y,\cdot)$ is concave for some integer $i$, it is sufficient to replace the primary unknown $(y^i,z^i)$ with the new pair of unknown variables  $(-y^i, -z^i)$  in the underlying system of BSDEs.

For each fixed $m,p\geq 1$ and $\theta\in (0,1)$, define
$$
\delta_{\theta}Y^{(m,p)}:=\frac{Y^{(m+p)}-\theta Y^{(m)}}{1-\theta}\ \ \ {\rm and}\ \ \ \delta_{\theta}Z^{(m,p)}:=\frac{Z^{(m+p)}-\theta Z^{(m)}}{1-\theta}.
$$
Then $(\delta_{\theta}Y^{(m,p)},\delta_{\theta}Z^{(m,p)})\in \mathcal{E}(\R^n)\times \mcal(\R^{n\times d})$ solves the following system of BSDEs: for each $i=1,\cdots, n$,
\begin{equation}\label{eq:5.14}
  \delta_{\theta} Y_t^{(m,p);i}=\xi^i+\int_t^T \delta_{\theta} g^{(m,p);i}(s,\delta_{\theta} Z_s^{(m,p);i}){\rm d}s-\int_t^T \delta_{\theta} Z_s^{(m,p);i}{\rm d}B_s, \ \ t\in\T,
\end{equation}
where $\ass$, for each $z\in \R^{1\times d}$,
\begin{equation}\label{eq:5.15}
\delta_{\theta} g^{(m,p);i}(s, z):={1\over 1-\theta}\left(g^{i}(s,Y_s^{(m+p-1)}, (1-\theta)z+\theta Z_s^{(m);i})-\theta g^{i}(s,Y_s^{(m-1)},Z_s^{(m);i})\right).
\end{equation}
It follows from \eqref{eq:5.15} and assumptions \ref{A:B2}, \ref{A:B3} and \ref{A:B1} that $\ass$, for each $z\in \R^{1\times d}$,
$$
\begin{array}{lll}
\Dis \delta_{\theta} g^{(m,p);i}(s,z)&\leq &\beta |\delta_{\theta}Y_s^{(m-1,p)}|+\beta |Y_s^{(m-1)}|+g^{i}(s,Y_s^{(m-1)},z) \vspace{0.1cm}\\
&\leq & \Dis \alpha_s+\beta |\delta_{\theta}Y_s^{(m-1,p)}|+2\beta |Y_s^{(m-1)}|+\frac{\gamma}{2}|z|^2,
\end{array}
$$
which together with \eqref{eq:5.3} means that all the conditions in \cref{lem:A.4} are satisfied for BSDE \eqref{eq:5.14}, and then for $i=1,\cdots,n$, we have
\begin{equation}\label{eq:5.16}
\begin{array}{ll}
&\Dis \exp\left(\gamma \left(\delta_{\theta} Y_t^{(m,p);i}\right)^+\right)\vspace{0.1cm}\\
\leq &\Dis \E_t\left[\exp\left(\gamma (\xi^i)^+ +\gamma\int_t^T \left(\alpha_s+\beta |\delta_{\theta}Y_s^{(m-1,p)}|+2\beta |Y_s^{(m-1)}|\right) {\rm d}s\right)\right],\ t\in\T.
\end{array}
\end{equation}
On the other hand, define
$$
\delta_{\theta}\widetilde Y^{(m,p)}:=\frac{Y^{(m)}-\theta Y^{(m+p)}}{1-\theta}\ \ \ {\rm and}\ \ \ \delta_{\theta}\widetilde Z^{(m,p)}:=\frac{Z^{(m)}-\theta Z^{(m+p)}}{1-\theta}.
$$
The same computation as above yields that for $i=1,\cdots,n$,
\begin{equation}\label{eq:5.17}
\begin{array}{ll}
&\Dis \exp\left(\gamma \left(\delta_{\theta} \widetilde Y_t^{(m,p);i}\right)^+\right)\vspace{0.1cm}\\
\leq &\Dis \E_t\left[\exp\left(\gamma (\xi^i)^+ +\gamma\int_t^T \left(\alpha_s+\beta |\delta_{\theta} \widetilde Y_s^{(m-1,p)}|+2\beta |Y_s^{(m+p-1)}|\right) {\rm d}s\right)\right],\ t\in\T.\vspace{0.2cm}
\end{array}
\end{equation}
Furthermore, observe that for $i=1,\cdots,n$ and $t\in\T$, we have
$$
\begin{array}{lll}
\Dis \left(\delta_{\theta} Y_t^{(m,p);i}\right)^- &=& \Dis \frac{\left(Y_t^{(m+p);i}-\theta Y_t^{(m);i}\right)^-} {1-\theta}= \Dis \frac{\left(\theta Y_t^{(m);i}-Y_t^{(m+p);i}\right)^+} {1-\theta} \vspace{0.1cm}\\
&\leq & \Dis \frac{\theta\left( Y_t^{(m);i}-\theta Y_t^{(m+p);i}\right)^+ +(1-\theta^2)|Y_t^{(m+p);i}|} {1-\theta}\vspace{0.1cm}\\
&\leq & \Dis \left(\delta_{\theta} \widetilde Y_t^{(m,p);i}\right)^+ +2|Y_t^{(m+p)}|
\end{array}
$$
and, similarly,
$$
\left(\delta_{\theta} \widetilde Y_t^{(m,p);i}\right)^- \leq \left(\delta_{\theta} Y_t^{(m,p);i}\right)^+ +2|Y_t^{(m)}|.
$$
It follows from \eqref{eq:5.16} and \eqref{eq:5.17} together with Jensen's inequality that for each $i=1,\cdots, n$ and $t\in\T$,
$$
\begin{array}{ll}
&\Dis \exp\left(\gamma |\delta_{\theta} Y_t^{(m,p);i}|\right) = \exp\left(\gamma \left(\delta_{\theta} Y_t^{(m,p);i}\right)^+\right)\cdot \exp\left(\gamma \left(\delta_{\theta} Y_t^{(m,p);i}\right)^-\right) \vspace{0.1cm}\\
\leq &\Dis
\E_t\left[\exp\left(2\gamma |\xi|+2\gamma |Y_t^{(m+p)}| +2\gamma\int_t^T \left(\alpha_s+2\beta |Y_s^{(m-1)}|+2\beta |Y_s^{(m+p-1)}|\right) {\rm d}s \right.\right. \vspace{0.1cm}\\
&\Dis \hspace{1.4cm}\left.\left. +2\gamma\beta \int_t^T \left( |\delta_{\theta}Y_s^{(m-1,p)}|+|\delta_{\theta}\widetilde Y_s^{(m-1,p)}| \right) {\rm d}s\right)\right]
\end{array}
$$
and
$$
\begin{array}{ll}
&\Dis \exp\left(\gamma |\delta_{\theta} \widetilde Y_t^{(m,p);i}|\right) = \exp\left(\gamma \left(\delta_{\theta} \widetilde Y_t^{(m,p);i}\right)^+\right)\cdot \exp\left(\gamma \left(\delta_{\theta} \widetilde Y_t^{(m,p);i}\right)^-\right) \vspace{0.1cm}\\
\leq &\Dis
\E_t\left[\exp\left(2\gamma |\xi|+2\gamma |Y_t^{(m)}| +2\gamma\int_t^T \left(\alpha_s+2\beta |Y_s^{(m-1)}|+2\beta |Y_s^{(m+p-1)}|\right) {\rm d}s \right.\right. \vspace{0.1cm}\\
&\Dis \hspace{1.4cm}\left.\left. +2\gamma\beta \int_t^T \left( |\delta_{\theta}Y_s^{(m-1,p)}|+|\delta_{\theta}\widetilde Y_s^{(m-1,p)}| \right) {\rm d}s\right)\right].
\end{array}
$$
Consequently, by Jensen's inequality again we have for each $t\in \T$,
$$
\begin{array}{ll}
&\Dis \exp\left(\gamma \left(|\delta_{\theta} Y_t^{(m,p);i}|+|\delta_{\theta} \widetilde Y_t^{(m,p);i}|\right) \right) \vspace{0.1cm}\\
\leq &\Dis
\E_t\left[\exp\left\{4\gamma \left(|\xi|+|Y_t^{(m)}| + |Y_t^{(m+p)}|+\int_t^T \left(\alpha_s+2\beta |Y_s^{(m-1)}|+2\beta |Y_s^{(m+p-1)}|\right) {\rm d}s\right) \right.\right. \vspace{0.1cm}\\
&\Dis \hspace{1.4cm}\left.\left. +4\gamma\beta \int_t^T \left( |\delta_{\theta}Y_s^{(m-1,p)}|+|\delta_{\theta}\widetilde Y_s^{(m-1,p)}| \right) {\rm d}s\right\}\right],\ \ \ \ i=1,\cdots, n,
\end{array}
$$
and then
\begin{equation}\label{eq:5.18}
\begin{array}{ll}
&\Dis \exp\left(\gamma \left(|\delta_{\theta} Y_t^{(m,p)}|+|\delta_{\theta} \widetilde Y_t^{(m,p)}|\right) \right) \vspace{0.1cm}\\
\leq &\Dis
\E_t\left[\exp\left\{4n\gamma \left(|\xi|+|Y_t^{(m)}| + |Y_t^{(m+p)}|+\int_t^T \left(\alpha_s+2\beta |Y_s^{(m-1)}|+2\beta |Y_s^{(m+p-1)}|\right) {\rm d}s\right) \right.\right. \vspace{0.1cm}\\
&\Dis \hspace{1.4cm}\left.\left. +4n\gamma\beta \int_t^T \left( |\delta_{\theta}Y_s^{(m-1,p)}|+|\delta_{\theta}\widetilde Y_s^{(m-1,p)}| \right) {\rm d}s\right\}\right].
\end{array}
\end{equation}
In view of \eqref{eq:5.18}, Doob's maximal inequality for martingales together with H\"{o}lder's inequality yields that for each $q>1$ and $t\in\T$, we have
\begin{equation}\label{eq:5.19}
\begin{array}{ll}
&\Dis \E\left[\exp\left(q\gamma \sup_{s\in [t,T]}\left(|\delta_{\theta} Y_s^{(m,p)}|+|\delta_{\theta} \widetilde Y_s^{(m,p)}|\right)\right)\right]\vspace{0.2cm}\\
\leq &\Dis \left(q\over q-1\right)^q\E\left[\exp\left\{4nq\gamma \left(|\xi|+\sup_{s\in [t,T]} \left(|Y_s^{(m)}|+|Y_s^{(m+p)}|\right)\right)\right.\right.
\vspace{0.1cm}\\
&\Dis \hspace{1.5cm} +4nq\gamma\int_t^T \left(\alpha_s+2\beta |Y_s^{(m-1)}|+2\beta |Y_s^{(m+p-1)}|\right) {\rm d}s\vspace{0.1cm}\\
&\Dis \hspace{1.5cm} \left.\left.+4nq\gamma \beta \sup_{s\in [t,T]}\left( |\delta_{\theta}Y_s^{(m-1,p)}|+|\delta_{\theta}\widetilde Y_s^{(m-1,p)}| \right)(T-t) \right\}\right]\vspace{0.2cm}\\
\leq &\Dis [\bar C(q)]^{1/2}\left\{\E\left[\exp\left(8nq\gamma \beta \sup_{s\in [t,T]}\left( |\delta_{\theta}Y_s^{(m-1,p)}|+|\delta_{\theta}\widetilde Y_s^{(m-1,p)}| \right) (T-t) \right)\right]\right\}^{1/2},
\end{array}
\end{equation}
where, in view of \eqref{eq:5.3},
$$
\begin{array}{lll}
\bar C(q) & :=&\Dis A(q)\sup_{m,p\geq 1}\E\left[\exp\left\{8nq\gamma \left(|\xi|+\sup_{t\in\T}|Y_t^{(m)}|+\sup_{t\in\T} |Y_t^{(m+p)}|\right)\right.\right.\vspace{0.1cm}\\
&& \hspace{2.1cm} \Dis +\left. \left. 8nq\gamma\int_0^T \left(\alpha_s+2\beta |Y_s^{(m-1)}|+2\beta |Y_s^{(m+p-1)}|\right) {\rm d}s\right\}\right]<+\infty\vspace{0.1cm}
\end{array}
$$
with $A(q)$ being defined in \eqref{eq:5.3*}. 

Based on the above analysis, we now can prove that $\{(Y^{(m)},Z^{(m)})\}_{m=1}^{\infty}$ is a Cauchy sequence in the space $\s^q(\R^n)\times \hcal^q(\R^{n\times d})$ for each $q>1$. In fact, observing the similarity of \eqref{eq:5.19} and \eqref{eq:5.6}, we can induce with respect to $m$ and use a similar argument as that obtaining \eqref{eq:5.3} to derive the existence of a positive constant $\bar K(q)$ depending on $q$ and being independent of $\theta$ such that for each $q>1$,
$$
\begin{array}{ll}
& \Dis \E\left[\exp\left(q\gamma \sup_{t\in \T}\left(|\delta_{\theta} Y_t^{(m,p)}|+|\delta_{\theta} \widetilde Y_t^{(m,p)}|\right)\right)\right]\vspace{0.1cm}\\
\leq & \Dis \bar K(q) \left(\E\left[\exp\left(q\gamma \sup_{t\in \T}\left(|\delta_{\theta} Y_t^{(1,p)}|+|\delta_{\theta} \widetilde Y_t^{(1,p)}|\right)\right)\right]\right)^{\frac{2^{[2n\beta T]+1}}{2^m}},
\end{array}
$$
from which together with \eqref{eq:5.3} it follows that for each $q>1$ and $\theta\in (0,1)$,
\begin{equation}\label{eq:5.20}
\limsup_{m\To\infty}\sup_{p\geq 1}\E\left[\exp\left(q\gamma \sup_{t\in \T}\left(|\delta_{\theta} Y_t^{(m,p)}|+|\delta_{\theta} \widetilde Y_t^{(m,p)}|\right)\right)\right]\leq \bar K(q).
\end{equation}
Thus, for each $\theta\in (0,1)$,
$$
\limsup_{m\To\infty}\sup_{p\geq 1}\E\left[\sup_{t\in \T}|Y_t^{(m+p)}-\theta Y_t^{(m)}|\right]\leq (1-\theta)\frac{\bar K(2)}{2\gamma},
$$
and then, in view of \eqref{eq:5.3},
$$
\limsup_{m\To\infty}\sup_{p\geq 1}\E\left[\sup_{t\in \T}|Y_t^{(m+p)}- Y_t^{(m)}|\right]\leq (1-\theta)\left(\frac{\bar K(2)}{2\gamma}+\sup_{m\geq 1}\E\left[\sup_{t\in \T}|Y_t^{(m)}|\right]\right)<+\infty.
$$
Sending $\theta$ to $1$, in view of  \eqref{eq:5.3},  we see  that there is an adapted process $Y\in \mathcal{E}(\R^n)$ such that for each $q>1$,
\begin{equation}\label{eq:5.21}
\lim_{m\To\infty} \E\left[\sup_{t\in \T}|Y_t^{(m)}-Y_t|^q\right]=0
\end{equation}
and
\begin{equation}\label{eq:5.22}
\lim_{m\To\infty} \E\left[\exp\left(q\sup_{t\in \T}|Y_t^{(m)}-Y_t|\right)\right]=1.\vspace{0.2cm}
\end{equation}
Furthermore, it follows from It\^{o}'s formula that for each $m,p\geq 1$,
\begin{equation}\label{eq:5.23}
\begin{array}{l}
\Dis \E\left[\int_0^T |Z^{(m+p)}_s-Z^{(m)}_s|^2 {\rm d}s\right]\leq 2\E\left[\sup_{t\in\T}\left|Y^{(m+p)}_t-Y^{(m)}_t\right|\right.  \vspace{0.2cm}\\
\Dis \hspace{2cm} \cdot\left.\int_0^T \sum_{i=1}^n\left|g^i(s,Y^{(m+p-1)}_s, Z^{(m+p);i}_s)-g^i(s,Y^{(m-1)}_s, Z^{(m);i}_s)\right| {\rm d}s \right].
\end{array}
\end{equation}
And, by virtue of \ref{A:B1}, \eqref{eq:5.3} and \eqref{eq:5.12} we get that
\begin{equation}\label{eq:5.24}
\sup_{m,p\geq 1}\E\left[\left(\int_0^T \sum_{i=1}^n \left|g^i(s,Y^{(m+p-1)}_s, Z^{(m+p);i}_s)-g^i(s,Y^{(m-1)}_s, Z^{(m);i}_s)\right| {\rm d}s\right)^2\right]<+\infty.
\end{equation}
Then, applying H\"{o}lder's inequality to \eqref{eq:5.23} and using \eqref{eq:5.21} and \eqref{eq:5.24} leads to that
$$
\lim_{m\To\infty}\sup_{p\geq 1}\E\left[\int_0^T |Z_s^{(m+p)}-Z_s^{(m)}|^2{\rm d}s\right]=0,
$$
from which together with \eqref{eq:5.12} it follows that there exists a process $Z\in \mcal(\R^{n\times d})$ such that
\begin{equation}\label{eq:5.25}
\RE\ q>1,\ \ \lim_{m\To\infty} \E\left[\left(\int_0^T |Z_s^{(m)}-Z_s|^2{\rm d}s\right)^{q/2}\right]=0.
\end{equation}
Finally, in view of \eqref{eq:5.21}, \eqref{eq:5.22} and \eqref{eq:5.25}, by sending $m$ to infinity in \eqref{eq:5.2} we can deduce that $(Y,Z)$ is a desired solution of system of BSDE \eqref{eq:1.1}. \vspace{0.2cm}

It remains to show the uniqueness part for completing the proof of \cref{thm:2.8}. For this, let $(\widetilde Y,\widetilde Z)$ be also a solution of system of BSDE \eqref{eq:1.1} in the space $\mathcal{E}(\R^n)\times\mcal(\R^{n\times d})$, and for $\theta\in (0,1)$, define
$$
\delta_{\theta}U:=\frac{Y-\theta \widetilde Y}{1-\theta},\ \ \ \ \ \ \delta_{\theta}V:=\frac{Z-\theta \widetilde Z}{1-\theta},
$$
and
$$
\delta_{\theta}\widetilde U:=\frac{\widetilde Y-\theta Y}{1-\theta},\ \ \ \ \ \ \delta_{\theta}\widetilde V:=\frac{\widetilde Z-\theta Z}{1-\theta}.\vspace{0.1cm}
$$
Using a similar argument as that from \eqref{eq:5.14} to \eqref{eq:5.19}, we can deduce that for each $q>1$,
\begin{equation}\label{eq:5.26}
\begin{array}{ll}
&\Dis \E\left[\exp\left(q\gamma \sup_{s\in [t,T]}\left(|\delta_{\theta} U_s|+|\delta_{\theta} \widetilde U_s| \right)\right)\right]\vspace{0.2cm}\\
\leq &\Dis \left(q\over q-1\right)^q  \E\left[\exp\left\{4nq\gamma \left(|\xi|+\sup_{s\in [t,T]}\left(|Y_s|+|\widetilde Y_s|\right)+\int_t^T \left(\alpha_s+2\beta |Y_s|+2\beta |\widetilde Y_s|\right) {\rm d}s\right)\right.\right.\vspace{0.1cm}\\
&\Dis \hspace{3.1cm}\left.\left. +4nq\gamma \beta \sup_{s\in [t,T]}\left(|\delta_{\theta} U_s|+|\delta_{\theta} \widetilde U_s| \right)(T-t) \right\}\right]\vspace{0.2cm}\\
\leq &\Dis \left[\widetilde C(q)\right]^{1/2}\left\{\E\left[\exp\left(8nq\gamma \beta \sup_{s\in [t,T]}\left(|\delta_{\theta} U_s|+|\delta_{\theta} \widetilde U_s|\right) (T-t)\right)\right]\right\}^{1/2},\ \ t\in\T,
\end{array}
\end{equation}
where
$$
\begin{array}{l}
\widetilde C(q):=A(q)\E\left[\exp\left\{8nq\gamma \left(|\xi|+\sup_{t\in\T} \left(|Y_t|+|\widetilde Y_t|\right)+\int_0^T \left(\alpha_s+2\beta |Y_s|+2\beta |\widetilde Y_s|\right){\rm d}s\right) \right\}\right]\\
\hspace*{1cm} <+\infty
\end{array} \vspace{0.1cm}
$$
with $A(q)$ being defined in \eqref{eq:5.3*}. For $\beta=0$, it is clear from \eqref{eq:5.26} that
$$
\E\left[2\gamma \sup_{t\in [0,T]}|\delta_{\theta} U_t|\right] \leq \E\left[\exp\left(2\gamma \sup_{t\in [0,T]}|\delta_{\theta} U_t|\right)\right]\leq \sqrt{\widetilde C(2)},\vspace{0.2cm}
$$
and then
$$
\begin{array}{lll}
\Dis \E\left[\sup_{t\in [0,T]}|Y_t-\widetilde Y_t|\right]& \leq &\Dis \E\left[\sup_{t\in [0,T]}|Y_t-\theta \widetilde Y_t|\right]+(1-\theta)\E\left[\sup_{t\in [0,T]}|\widetilde Y_t|\right]\vspace{0.2cm}\\
&\leq & \Dis (1-\theta)\left(\frac{\sqrt{\widetilde C(2)}}{2\gamma}+\E\left[\exp\left(\sup_{t\in [0,T]}|\widetilde Y_t|\right)\right]\right),
\end{array}
$$
in which letting $\theta\To 1$ yields that $Y=\widetilde Y$ and then $Z=\widetilde Z$ on the time interval $\T$. Otherwise, in the case of $T\leq \bar\eps:=1/8n\beta$, it follows from \eqref{eq:5.26} that
$$
\E\left[2\gamma \sup_{t\in [0,T]}|\delta_{\theta} U_t|\right] \leq \E\left[\exp\left(2\gamma \sup_{t\in [0,T]}|\delta_{\theta} U_t|\right)\right]\leq \widetilde C(2),
$$
and then $Y=\widetilde Y$ and $Z=\widetilde Z$ on the time interval $[0,T]$. Similarly, if $m_0$ is the unique positive integer such that $(m_0-1)\bar\eps<T\leq m_0\bar\eps$, then we can successively prove the uniqueness on
the time intervals $[T-\bar\eps,T]$, $[T-2\bar\eps,T-\bar\eps]$, $\cdots$, $[T-(m_0-1)\bar\eps,T-(m_0-2)\bar\eps]$ and $[0,T-(m_0-1)\bar\eps]$.\vspace{0.2cm}

The proof of \cref{thm:2.8} is then complete.

\appendix
\section{Several auxiliary results on one-dimensional quadratic BSDEs}
\renewcommand{\appendixname}{}

We consider the following one-dimensional BSDE:
\begin{equation}\label{eq:A.1}
  Y_t=\eta+\int_t^T f(s,Z_s)\, {\rm d}s-\int_t^T Z_s\,  {\rm d}B_s, \ \ t\in\T,
\end{equation}
where the terminal value $\eta$ is a real-valued $\F_T$-measurable random variable, and the generator function $f(\cdot, \cdot, z):\Omega\times\T\To \R$
is $(\F_t)$-progressively measurable for each $z\in \R^{1\times d}$. Here,  the solution $(Y_t,Z_t)_{t\in\T}$ is defined as a pair of $(\F_t)$-progressively measurable processes taking values in $\R\times\R^{1\times d}$,  such that \eqref{eq:A.1} is satisfied.

Assume that there exists a pair of processes $(U,V)\in \s^\infty(\R^n)\times {\rm BMO}(\R^{n\times d})$ such that the generator $f$  satisfies the following assumptions.

\begin{enumerate}
\renewcommand{\theenumi}{(A\arabic{enumi})}
\renewcommand{\labelenumi}{\theenumi}
\item\label{A:A1} $\as$, we have
    $$
    |f(\omega,t,z)|\leq \alpha_t(\omega)+\phi(|U_t(\omega)|)+n\lambda |V_t(\omega)|^{1+\delta}+\frac{\gamma}{2} |z|^2 \quad \hbox{\rm for each $z\in \R^{1\times d}$} ;
    $$

\item\label{A:A2} $\as$, we have
    $$
     |f(\omega,t,z)-f(\omega,t,\bar z)| \leq \phi(|U_t(\omega)|)\left(1+2|V_t(\omega)|+|z|+|\bar z|\right) |z-\bar z|
    $$
 for each $(z, \bar z)\in (\R^{1\times d})^2$.\vspace{0.1cm}
\end{enumerate}

The following lemma slightly generalizes \citet[Lemma 2.1]{HuTang2016SPA}.

\begin{lem}\label{lem:A.1}
Let the generator $f$ satisfy assumptions \ref{A:A1} and \ref{A:A2}, and
both $|\eta|$ and $\int_0^T \alpha_t{\rm d}t$ be (essentially) bounded. Then, BSDE \eqref{eq:A.1} admits a unique solution $(Y,Z)$ such that $Y$ is (essentially) bounded and $Z\cdot B:=\left(\int_0^t Z_s{\rm d}B_s\right)_{t\in\T}$ is a BMO martingale. Moreover, for each $t\in\T$ and each stopping time $\tau$ with values in $[t,T]$, we have
\begin{equation}\label{eq:A.2}
\begin{array}{lll}
|Y_t|&\leq & \Dis {1\over \gamma}\ln 2+\|\eta\|_{\infty}+\left\|\int_0^T \alpha_s{\rm d}s\right\|_{\infty}\vspace{0.1cm}\\
&&\Dis +\phi\left(\|U\|_{\s^{\infty}_{[t,T]}}\right)(T-t)
+\gamma^{\frac{1+\delta}{1-\delta}}C_{\delta,\lambda,n} \|V\|_{{\rm BMO}_{[t,T]}}^{2\frac{1+\delta}{1-\delta}}(T-t)
\end{array}
\end{equation}
and
\begin{equation}\label{eq:A.3}
\begin{array}{lll}
&& \Dis\E_\tau\left[\int_\tau^T |Z_s|^2{\rm d}s\right]\vspace{0.1cm}\\
&\leq &\Dis {1\over \gamma^2} \exp(2\gamma \|\eta\|_\infty)+{1\over \gamma}\exp\left(2\gamma \left\|\sup_{s\in [t,T]}|Y_s|\right\|_\infty\right)\vspace{0.2cm}\\
&& \cdot \Dis\left(1+2\left\|\int_0^T \alpha_s {\rm d}s\right\|_{\infty}+2\phi\left(\|U\|_{\s^{\infty}_{[t,T]}}\right)(T-t)
+2C_{\delta,\lambda,n}\|V\|_{{\rm BMO}_{[t,T]}}^{2\frac{1+\delta}{1-\delta}}(T-t)\right),
\end{array}
\end{equation}
where
\begin{equation}\label{eq:A.4}
C_{\delta,\lambda,n}:=\frac{1-\delta}{2}(1+\delta)^{\frac{1+\delta}{1-\delta}}
(n\lambda)^{\frac{2}{1-\delta}}.\vspace{0.2cm}
\end{equation}
\end{lem}

\begin{proof} Since $V\in {{\rm BMO}}(\R^{n\times d})$, it follows from Young's inequality that for each real $k>0$,
\begin{equation}\label{eq:A.5}
\begin{array}{lll}
\Dis kn\lambda |V_s|^{1+\delta} &=& \Dis \left( (1+\delta)^{\frac{1+\delta}{1-\delta}}\|V\|_{{\rm BMO}_{[t,T]}}^{2\frac{1+\delta}{1-\delta}} (kn\lambda)^{\frac{2}{1-\delta}} \right)^{\frac{1-\delta}{2}}\left(\frac{|V_s|^2}{(1+\delta)\|V\|_{{\rm BMO}_{[t,T]}}^2} \right)^{\frac{1+\delta}{2}}\vspace{0.2cm}\\
&\leq & \Dis \frac{1}{2\|V\|_{{\rm BMO}_{[t,T]}}^2}|V_s|^2+ k^{\frac{2}{1-\delta}} C_{\delta,\lambda,n} \|V\|_{{\rm BMO}_{[t,T]}}^{2\frac{1+\delta}{1-\delta}},\ \ \  0\leq t\leq s\leq T,
\end{array}
\end{equation}
where the constant $C_{\delta,\lambda,n}$ is defined in \eqref{eq:A.4}. On the other hand, it follows from John-Nirenberg inequality for BMO martingale (see for example Lemma A.1 in \citet{HuTang2016SPA}) that for each $t\in\T$ and each stopping time $\tau$ with values in $[t,T]$,
\begin{equation}\label{eq:A.6}
\E_\tau\left[\exp\left(\frac{1}{2\|V\|_{{\rm BMO}_{[t,T]}}^2}\int_\tau^T |V_s|^2{\rm d}s\right)\right]\leq \frac{1}{1-{1\over 2}}=2,\ \ t\in \T.
\end{equation}
Thus, combining previous inequality and inequality \eqref{eq:A.5} with $k:=p\gamma$ yields that for each $p\geq 1$ and $t\in\T$,
\begin{equation}\label{eq:A.7}
\E_t\left[\exp\left(p\gamma n\lambda\int_t^T |V_s|^{1+\delta}{\rm d}s\right)\right]\leq 2\exp\left( (p\gamma)^{\frac{2}{1-\delta}} C_{\delta,\lambda,n} \|V\|_{{\rm BMO}_{[t,T]}}^{2\frac{1+\delta}{1-\delta}}(T-t)\right)<+\infty,
\end{equation}
and then, in view of $U\in \s^{\infty}(\R^n)$ and the boundedness of $|\eta|$ and $\int_0^T \alpha_s{\rm d}s$,
$$
\begin{array}{ll}
&\Dis \E_t\left[\exp\left(p\gamma |\eta|+p\gamma\int_t^T \left(\alpha_s+\phi(|U_s|)+n\lambda |V_s|^{1+\delta}\right){\rm d}s\right)\right]\vspace{0.1cm}\\
\leq &\Dis 2\exp\left(p\gamma \|\eta\|_{\infty}+p\gamma\left\|\int_0^T \alpha_s{\rm d}s\right\|_{\infty}+p\gamma \phi\left(\|U\|_{\s^{\infty}_{[t,T]}}\right)(T-t)\right.\vspace{0.1cm}\\
&\Dis \left. \hspace{1.2cm}+(p\gamma)^{\frac{2}{1-\delta}} C_{\delta,\lambda,n} \|V\|_{{\rm BMO}_{[t,T]}}^{2\frac{1+\delta}{1-\delta}}(T-t)\right)<+\infty.
\end{array}
$$
In view of assumptions \ref{A:A1}-\ref{A:A2} and  with previous inequality in the hand, we can apply Theorem 2 in \citet{BriandHu2008PTRF} to see that BSDE \eqref{eq:A.1}
admits a solution $(Y,Z)$ such that
$$
\E\left[\int_0^T |Z_s|^2{\rm d}s\right]<+\infty
$$
and
$$
\begin{array}{lll}
\Dis \exp\left(\gamma |Y_t|\right)
&\leq &\Dis 2\exp\left(\gamma \|\eta\|_{\infty}+\gamma\left\|\int_0^T \alpha_s{\rm d}s\right\|_{\infty}+\gamma \phi\left(\|U\|_{\s^{\infty}_{[t,T]}}\right)(T-t)\right.\vspace{0.1cm}\\
&&\Dis \hspace{1.2cm}\left. +\gamma^{\frac{2}{1-\delta}} C_{\delta,\lambda,n} \|V\|_{{\rm BMO}_{[t,T]}}^{2\frac{1+\delta}{1-\delta}}(T-t)\right).
\end{array}
$$
This shows that \eqref{eq:A.2} holds and $Y$ is bounded.\vspace{0.2cm}

We now show that \eqref{eq:A.3} holds and $Z\cdot B$ is a BMO martingale. Using It\^{o}-Tanaka's formula to compute $\exp(2\gamma |Y_t|)$ and utilizing assumption \ref{A:A1}, we have, for each $t\in\T$ and each stopping time $\tau$ with values in $[t,T]$,
$$
\begin{array}{ll}
&\Dis \exp(2\gamma |Y_\tau|)+2\gamma^2\E_\tau\left[\int_\tau^T \exp(2\gamma |Y_s|) |Z_s|^2{\rm d}s\right]\vspace{0.1cm}\\
\leq &\Dis \E_\tau\left[\exp(2\gamma |\eta|)\right]+2\gamma\E_\tau\left[\int_\tau^T \exp(2\gamma |Y_s|)\left( \alpha_s+\phi(|U_s|)+n\lambda |V_s|^{1+\delta}+\frac{\gamma}{2}|Z_s|^2\right){\rm d}s\right].
\end{array}
$$
Therefore, in view of \eqref{eq:A.5} with $k:=1$,
$$
\begin{array}{lll}
&&\Dis \gamma^2\E_\tau\left[\int_\tau^T |Z_s|^2{\rm d}s\right]\vspace{0.1cm}\\
&\leq &\Dis \exp(2\gamma \|\eta\|_\infty)+2\gamma\exp\left(2\gamma \left\|\sup_{s\in [t,T]}|Y_s|\right\|_\infty\right)\vspace{0.1cm}\\
&& \cdot \Dis\left(\left\|\int_0^T \alpha_s {\rm d}s\right\|_{\infty}+\phi\left(\|U\|_{\s^{\infty}_{[t,T]}}\right)(T-t)+{1\over 2}+C_{\delta,\lambda,n} \|V\|_{{\rm BMO}_{[t,T]}}^{2\frac{1+\delta}{1-\delta}}(T-t)\right)<+\infty,
\end{array}
$$
from which the desired conclusions follows immediately.

Finally, in view of assumption \ref{A:A2} with $(U,V)\in \s^{\infty}(\R^n)\times {{\rm BMO}}(\R^{n\times d})$, by a similar argument to that in \citet[ Lemma 2.1 ]{HuTang2016SPA},  we can use the Girsanov transform to prove a comparison result on the solutions of BSDE \eqref{eq:A.1}, which yields the desired uniqueness.
\end{proof}

\begin{lem}\label{lem:A.2}
Assume that $(U,V)\in \s^\infty(\R^n)\times {\rm BMO}(\R^{n\times d})$, both $|\eta|$ and $\int_0^T \alpha_t{\rm d}t$ are (essentially) bounded, and $(Y,Z)$ is a solution of BSDE \eqref{eq:A.1} such that $Y$ is (essentially) bounded.

(i) If $\as$, it holds that
    \begin{equation}\label{eq:A.8}
    {\rm sgn} (Y_t(\omega))f\left(\omega,t,Z_t(\omega)\right)\leq \alpha_t(\omega)+\beta |U_t(\omega)|+\lambda |V_t(\omega)|^{1+\delta}+\frac{\gamma}{2} |Z_t(\omega)|^2,
    \end{equation}
    then for each $t\in\T$, we have
    $$
    \begin{array}{lll}
    \Dis \exp\left(\gamma |Y_t|\right)&\leq &\Dis \E_t\left[\exp\left(\gamma \|\eta\|_{\infty}+\gamma\left\|\int_0^T \alpha_s {\rm d}s\right\|_{\infty}\right.\right.\vspace{0.2cm}\\
    &&\Dis \hspace{1.5cm}\left.\left.+\beta\gamma \|U\|_{\s^{\infty}_{[t,T]}}(T-t)+\lambda\gamma \int_t^T |V_s|^{1+\delta} {\rm d}s\right)\right];
    \end{array}
    $$

(ii) If $\as$, it holds that
    \begin{equation}\label{eq:A.9}
    f\left(\omega,t,Z_t(\omega)\right)\geq \frac{\bar \gamma}{2} |Z_t(\omega)|^2-\alpha_t(\omega)-\beta |U_t(\omega)|-\lambda |V_t(\omega)|^{1+\delta}
    \end{equation}
    or
    \begin{equation}\label{eq:A.10}
    f\left(\omega,t,Z_t(\omega)\right)\leq -\frac{\bar \gamma}{2} |Z_t(\omega)|^2+\alpha_t(\omega)+\beta |U_t(\omega)|+\lambda |V_t(\omega)|^{1+\delta},\vspace{0.1cm}
    \end{equation}
    then for each $\eps\in (0,\frac{\bar\gamma}{9}]$ and $t\in \T$, we have
    $$
    \begin{array}{lll}
    \Dis \E_t\left[\exp\left(\frac{\bar \gamma}{2}\eps\int_t^T |Z_s|^2{\rm d}s \right)\right]&\leq & \Dis\E_t\left[\exp\left(6\eps\left\|\sup_{s\in [t,T]}|Y_s|\right\|_{\infty}+3\eps\left\|\int_0^T \alpha_s {\rm d}s\right\|_{\infty} \right.\right.\vspace{0.2cm}\\
    &&\Dis \hspace{1.4cm}\left.\left.+3\eps\beta \|U\|_{\s^{\infty}_{[t,T]}}(T-t)+3\eps\lambda \int_t^T |V_s|^{1+\delta} {\rm d}s\right)\right].
    \end{array}
    $$
\end{lem}

\begin{proof}
In view of \eqref{eq:A.7}, using It\^{o}-Tanaka's formula to compute $$\exp\left(\gamma |Y_t|+\gamma\int_0^t \left(\alpha_s(\omega)+\beta |U_s(\omega)|+\lambda |V_s(\omega)|^{1+\delta}\right){\rm d}s\right)$$
one can easily obtain (i). And, in view of \eqref{eq:A.7} again, we can
apply a similar argument as in the proof of   \citet[Proposition 2]{FanHuTang2019ArXiv} to get (ii). The detailed proof is omitted here.
\end{proof}

The following two lemmas provide some bounds on the (possibly unbounded) solutions of one-dimensional quadratic BSDEs, which can be derived from (i) of \citet[Proposition 1]{FanHuTang2019ArXiv}. We omit the detailed proof here.

\begin{lem}\label{lem:A.3}
Assume that there exists an $(\F_t)$-progressively measurable scalar-valued non-negative process $(\bar \alpha_t)_{t\in\T}$ such that $\as$,
$$
\RE\ z\in \R^{1\times d},\ \ \ |f(\omega,t,z)|\leq \bar \alpha_t(\omega)+\frac{\gamma}{2} |z|^2.
$$
Then, for any solution $(Y,Z)$ of BSDE \eqref{eq:A.1} satisfying
$$
\E\left[\exp\left(2\gamma\sup_{t\in \T}|Y_t|+2\gamma\int_0^T \bar \alpha_s{\rm d}s\right)\right]<+\infty,
$$
we have
$$
\exp\left(\gamma |Y_t|\right)\leq \E_t\left[\exp\left(\gamma |\eta|+\gamma\int_t^T \bar \alpha_s{\rm d}s\right)\right],\ \ t\in\T
$$
and
$$
\E\left[\int_0^T |Z_s|^2{\rm d}s\right]\leq {1\over \gamma^2} \E\left[\exp\left(2\gamma |\eta|+2\gamma\int_0^T \bar \alpha_s{\rm d}s\right)\right].\vspace{0.3cm}
$$
\end{lem}

\begin{lem}\label{lem:A.4}
Assume that there exists an $(\F_t)$-progressively measurable scalar-valued non-negative process $(\bar \alpha_t)_{t\in\T}$ such that $\as$,
$$
\RE\ z\in \R^{1\times d},\ \ \ f(\omega,t,z)\leq \bar \alpha_t(\omega)+\frac{\gamma}{2} |z|^2.
$$
Then, for any solution $(Y,Z)$ of BSDE \eqref{eq:A.1} satisfying
$$
\E\left[\exp\left(2\gamma\sup_{t\in \T}Y_t^+ +2\gamma\int_0^T \bar \alpha_s{\rm d}s\right)\right]<+\infty,
$$
we have
$$
\exp\left(\gamma Y_t^+\right)\leq \E_t\left[\exp\left(\gamma \eta^+ +\gamma\int_t^T \bar \alpha_s{\rm d}s\right)\right],\ \ t\in\T.
$$
\end{lem}



\setlength{\bibsep}{2pt}

\def\cprime{$'$}

\end{document}